\documentclass[10pt]{article}
\usepackage{verbatim,amsmath,amssymb,amsfonts,amsthm,amscd,graphicx,psfrag,epsfig,mathrsfs,stmaryrd}
\usepackage{tikz-cd}
\newtheorem{theorem}{Theorem}[section]

\newtheorem{theorem-definition}[theorem]{Theorem-Definition}
\newtheorem{theorem-construction}[theorem]{Theorem-Construction}
\newtheorem{lemma-definition}[theorem]{Lemma--Definition}
\newtheorem{lemma-construction}[theorem]{Lemma--Construction}
\newtheorem{lemma}[theorem]{Lemma}
\newtheorem{proposition}[theorem]{Proposition}
\newtheorem{corollary}[theorem]{Corollary}
\newtheorem{conjecture}[theorem]{Conjecture}
\theoremstyle{definition}
\newtheorem{definition}[theorem]{Definition}
\newtheorem{eg}[theorem]{Example}
\newtheorem{example}[theorem]{Example}

\newtheorem{remark}[theorem]{Remark}
\newcommand{\old}[1]{}

\newcommand{\D}{{\mathbb D}}

\newcommand{\R}{{\mathbb R}}

\renewcommand{\P}{{\mathbb P}}

\makeatletter
\newcommand{\extp}{\@ifnextchar^\@extp{\@extp^{\,}}}
\def\@extp^#1{\mathop{\bigwedge\nolimits^{\!#1}}}
\makeatother
\newcommand\restr[2]{{
		\left.\kern-\nulldelimiterspace 
		#1 
		\vphantom{\big|} 
		\right|_{#2} 
}}

\newcommand{\ra}{\rightarrow}
\newcommand{\be}{\begin{equation}}
	\newcommand{\ee}{\end{equation}}
\newcommand{\bt}{\begin{theorem}}
	
		\newcommand{\et}{\end{theorem}}
	\newcommand{\bd}{\begin{definition}}
		\newcommand{\ed}{\end{definition}}
	\newcommand{\bp}{\begin{proposition}}
		\newcommand{\ep}{\end{proposition}}
	
\newcommand{\bl}{\begin{lemma}}
	\newcommand{\el}{\end{lemma}}
\newcommand{\bc}{\begin{corollary}}
	\newcommand{\ec}{\end{corollary}}
\newcommand{\bcon}{\begin{conjecture}}
	\newcommand{\econ}{\end{conjecture}}
\newcommand{\la}{\label}
\usepackage{mathtools}

\usepackage{graphicx}
\usepackage{amsrefs}
\usepackage{geometry}
\usepackage{graphicx}
\usepackage{mathtools}
\usepackage{bbm}
\usepackage{bm}
\usepackage{amsmath,amsthm,amssymb,graphicx,mathtools,tikz-cd,hyperref}
\usepackage[]{subcaption}
\usetikzlibrary{positioning}
\usepackage{amssymb}

\newcommand{\red}[1]{{\color{red} #1}}
\newcommand{\blue}[1]{{\color{blue} #1}}

\newcommand{\snote}[1]{{\color{red} \sf [ #1 ]}}

\usepackage{blkarray}

\newcommand{\Cat}{\mathrm{Cat}}
\newcommand{\IG}{\mathrm{IG}}
\newcommand{\LG}{\mathrm{LG}}
\newcommand{\Gr}{\mathrm{Gr}}
\DeclareMathOperator{\Ker}{\mathrm{Ker}}
\DeclareMathOperator{\Span}{\mathrm{Span}}
\newcommand{\Karp}{\mathrm{Karp}}

\newcommand{\wt}{\widetilde}
\newcommand{\Alt}{\bigwedge\nolimits} 

\newenvironment{sbm}
{\left[ \begin{smallmatrix}
	}
	{ 
	\end{smallmatrix} \right]
}

\usepackage{blkarray}

\title{Electrical networks and Lagrangian Grassmannians}
\author{Sunita Chepuri, Terrence George and David E Speyer}
\date{}

\begin{document}
	
	\maketitle
	
	\begin{abstract}
Cactus networks were introduced by Lam as a generalization of planar electrical networks.  He defined a map from these networks to the Grassmannian Gr($n+1,2n$) and showed that the image of this map, $\mathcal X_n$ lies inside the totally nonnegative part of this Grassmannian.  In this paper, we show that $\mathcal X_n$ is exactly the elements of Gr($n+1,2n$) that are both totally nonnegative and isotropic for a particular skew-symmetric bilinear form.  For certain classes of cactus networks, we also explicitly describe how to turn response matrices and effective resistance matrices into points of Gr($n+1,2n$) given by Lam's map.  Finally, we discuss how our work relates to earlier studies of total positivity for Lagrangian Grassmannians.
\end{abstract}
	
	\section{Introduction}
	%
	%
	%
	%
	%
	%
	%
	%
	%

	This paper is motivated by the study of planar networks of electrical resistors.
	It builds on work of Curtis, Ingerman, and Morrow~\cite{CIM} and Lam~\cite{Lam18}, as well as older work, and explains how those ideas are clarified by thinking about total positivity in Lagrangian Grassmannians.  
	
	Let $G$ be a planar graph embedded in a disc $\D$ with $n$ vertices on the boundary of $\D$, labeled $1,2,\dots,n$, and a positive real number $c(e)$ associated to each edge $e$. We think of $G$ as a network of resistors, where the edge $e$ has conductance $c(e)$ (equivalently, resistance $1/c(e)$), and we will be interested in aspects of the network which are measurable by connecting batteries and electrical meters to these boundary vertices.  We will use the planar electrical network in Figure~\ref{EGFigure} as our running example.
	
	\begin{figure}
		\centerline{\includegraphics[width=0.3 \textwidth]{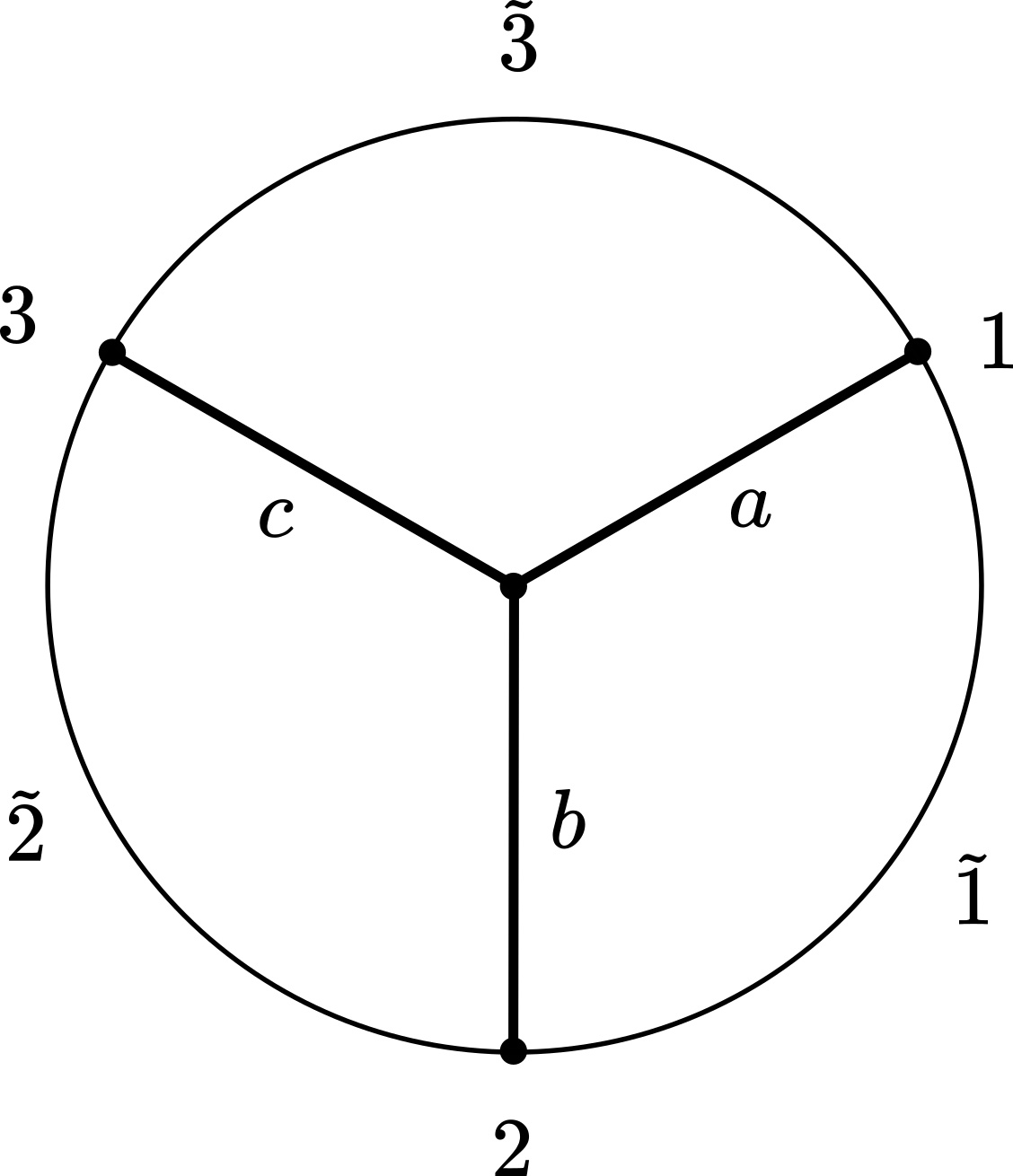}}
		\caption{Our running example of a planar electrical network} \label{EGFigure}.
	\end{figure}

	In particular, we can imagine placing vertex $i$ at voltage $V_i$ and measuring the resulting current $J_i$ flowing out of each vertex $i$ (if the current flows in, then $J_i$ is negative). 
	The map from the voltage vector $(V_1, V_2, \ldots, V_n)$ to the current vector $(J_1, J_2, \ldots, J_n)$ is linear, given by a symmetric matrix $L$ whose rows and columns sum to $0$; the matrix $L$ is called the \textit{response matrix}. (The reader who would like purely mathematical definitions of current, voltage and so forth should turn to Section~\ref{sec:harm-iso}.)
	We consider two planar networks to be electrically equivalent if they have the same response matrix.
	The classification of planar networks up to equivalence was carried out by Curtis, Ingerman, and Morrow~\cite{CIM}, building on work of de Verdi\`ere, Gitler, and Vertigan~\cite{CGV} and was then rewritten and improved by Lam~\cite{Lam}; we will describe their results shortly.

	We define a \textit{grove} of $G$ to be a subgraph $F$ of $G$ which contains every vertex of $G$, contains no cycles, and where every connected component of $F$ contains a vertex on the boundary of $\D$. The \textit{weight} of a grove $F$, written $w(F)$, is $\prod_{e \in F} c(e)$. For $i$ and $j$ distinct vertices on the boundary of $G$, we have (from \cites{Kirchoff,KW11})
	\begin{equation} L_{ij} = \frac{\sum_{F \ \text{has $n-1$ components, $i$ and $j$ in the same component}} w(F)}{\sum_{F \ \text{has $n$ components}} w(F)} . 
		\label{ResponseFormula} \end{equation}
	The value of $L_{ii}$ is determined by the condition that the rows and columns of $L$ sum to $0$.
	Another similar formula, due to Kirchhoff, describes the effective resistance $R_{ij}$ between vertices $i$ and $j$ as
	\begin{equation} R_{ij} =  \frac{\sum_{F \ \text{has two components, $i$ and $j$ in different components}} w(F)}{\sum_{F \ \text{connected}} w(F)}. 
		\label{ResistanceFormula} \end{equation}
	We set $R_{ii}=0$, so we may speak of the effective resistance matrix $R$.
	
	Motivated by Equations~\eqref{ResponseFormula} and~\eqref{ResistanceFormula}, we consider sums of $w(F)$ over groves $F$ with specified boundary connectivity. 
	We need some combinatorial notation first.
	
	We abbreviate $\{ 1,2,\ldots, n \}$ to $[n]$. A \textit{noncrossing partition} of $[n]$ is a set partition $\{ B_1, B_2, \ldots, B_k \}$ of $[n]$ which does not contain any two distinct blocks $B_i$ and $B_j$ with $a$, $c \in B_i$, with $b$, $d \in B_j$ and $a<b<c<d$.
	The number of noncrossing partitions of $[n]$ is the Catalan number  $\Cat_n := \tfrac{(2n)!}{n! (n+1)!}$. 
	Each grove $F$ defines a noncrossing partition $\sigma(F)$ of $[n]$ where $i$ and $j$ are in the same block of $\sigma(F)$ if and only if the boundary vertices $i$ and $j$ are in the same connected component of $F$. 
	For a non-crossing partition $\sigma$, define the \textit{grove measurements}
	\[ \Lambda_{\sigma} = \sum_{\sigma(F) = \sigma} w(F). \]
	It turns out \cite{Lam18}*{Proposition 4.4} that two networks are electrically equivalent if and only if the quantities $\Lambda_{\sigma}$ are proportional. 
	In other words, it is natural to coordinatize planar electrical networks using points in $\R \P^{\Cat_n -1 }$.
	
	\begin{eg}
		For the planar electrical network in Figure~\ref{EGFigure}, we have
		\[ \Lambda_{\{ 123 \}} = abc,\ \Lambda_{\{1\},\, \{23 \}} = bc,\ \Lambda_{\{2\},\, \{13 \}} = ac,\ \Lambda_{\{3\},\, \{12 \}} = ab,\ \Lambda_{\{1\},\, \{2 \},\, \{3 \}} = a+b+c . \]
		\[ L = \begin{bmatrix}
			\frac{-ab-ac}{a+b+c} & \frac{ab}{a+b+c} & \frac{ac}{a+b+c} \\[0.1 cm]
			\frac{ab}{a+b+c} & \frac{-ab-bc}{a+b+c} & \frac{bc}{a+b+c} \\[0.1 cm]
			\frac{ac}{a+b+c} &  \frac{bc}{a+b+c} &  \frac{-ac-bc}{a+b+c} \\[0.1 cm]
		\end{bmatrix} \qquad 
		R = \begin{bmatrix}
			0 & \frac{ac+bc}{abc} & \frac{ab+bc}{abc} \\[0.1 cm]
			\frac{ac+bc}{abc} & 0 & \frac{ab+ac}{abc} \\[0.1 cm]
			\frac{ab+bc}{abc} & \frac{ab+ac}{abc} & 0 \\[0.1 cm]
		\end{bmatrix}=
		\begin{bmatrix}
			0&\frac{1}{a}+\frac{1}{b}&\frac{1}{a}+\frac{1}{c} \\[0.1 cm]
			\frac{1}{a}+\frac{1}{b}&0&\frac{1}{b}+\frac{1}{c} \\[0.1 cm]
			\frac{1}{a}+\frac{1}{c}&\frac{1}{b}+\frac{1}{c}&0 \\[0.1 cm]
		\end{bmatrix}
		\]
	\end{eg}
	
	We now  preview a technical issue. 
	The subset of $\R \P^{\Cat_n -1 }$ corresponding to electrical networks is not closed.
	We can already see this in the case that $n=2$ and $G$ is a single edge from $1$ to $2$. If that single edge has conductance $c$, then $\Lambda_{\{1,2 \}} = c$ and $\Lambda_{\{1\}, \{2 \}} = 1$. So the subset of $\R \P^1$ corresponding to this network is the open interval $\{ [c:1] : 0 < c < \infty \}$. The limiting point $[0:1]$ can be achieved by taking $G$ to be graph with two vertices and no edges (in other words, deleting the lone edge from our initial graph). However, the limiting point $\{ [1:0] \}$ does not correspond to an electrical network. 
	Intuitively, this limiting point occurs when the conductance goes to $\infty$ or, in the language of electrical engineering, the boundary points $1$ and $2$ are ``shorted" to each other. 
	This idea motivates Lam's \text{cactus networks}, a generalization of electrical networks.
	Roughly speaking, a cactus network is a planar electrical network with some boundary vertices shorted together, so that the shorted vertices form a noncrossing partition of $[n]$; see Section~\ref{sec:cactus-nets} for details.
	By working with cactus networks rather than electrical networks, we obtain a closed subset of $\R \P^{\Cat_n-1}$. The reader should feel free to continue thinking of electrical networks for almost all purposes.

	Given a subset $I$ of $[n]$, and a non-crossing partition $\sigma$ of $[n]$, we say that $I$ and $\sigma$ are \text{concordant} if there is exactly one element of $I$ in each block of $\sigma$. (This notation differs from Lam, see Remark~\ref{rem:complement_warning}.)  Thus, Equations~\eqref{ResponseFormula} and~\eqref{ResistanceFormula} can be rewritten as
	\begin{equation} L_{ij} = \frac{\Lambda_{\{i,j\},\ \text{all other blocks singleton}}}{\Lambda_{\text{all singletons}}}
		\qquad R_{ij} = \frac{\sum_{\{ i, j \} \ \text{concordant to}\ \sigma} \Lambda_{\sigma}}{\Lambda_{\{1,2,\ldots, n \}}}
		\label{LambdaRatios} 
	\end{equation}
	
	Lam realized that it is valuable to work, not with non-crossing partitions of $[n]$, but with certain non-crossing partitions of $[2n]$. 
	Let $[\tilde{n}]$ be a second $n$ element index set, with elements $\{ \tilde{1}, \tilde{2}, \ldots, \tilde{n} \}$. We consider $[n] \sqcup [ \tilde{n} ]$ to be circularly ordered in the order $1$, $\tilde{1}$, $2$, $\tilde{2}$, \dots, $n$, $\tilde{n}$.  Given a non-crossing partition $\sigma$ of $[n]$ with $k$ blocks, there is a unique non-crossing partition $\tilde{\sigma}$ of $[\tilde{n}]$ with $n+1-k$ blocks such that $\sigma \sqcup \tilde{\sigma}$ forms a non-crossing partition of $[n] \sqcup [ \tilde{n} ]$. The non-crossing partition $\tilde{\sigma}$ is called the \textit{Kreweras complement} of $\sigma$~\cite{Kreweras}, and we will call $(\sigma, \tilde{\sigma})$ a \textit{Kreweras pair}.
	Given $I \subseteq [n]$ and $\tilde{I} \subseteq [\tilde{n}]$, and Krewaras pair $(\sigma, \tilde{\sigma})$,  we will say that $(I, \tilde{I})$ and $(\sigma, \tilde{\sigma})$ are concordant, if $I$ and $\sigma$ are concordant and $\tilde{I}$ and $\tilde{\sigma}$ are likewise concordant. So, in this case, $\# (I \sqcup \tilde{I}) = n+1$. 
	
	Consider the vector space $\R^{2n}$ with basis $e_1, e_{\wt{1}},\dots,e_n,e_{\wt{n}}$. For $I \sqcup \wt{I} \subseteq [n] \sqcup [\wt{n}]$, we put $e_{I, \wt{I}} = \Alt_{i \in I \sqcup \wt{I}} e_i$, where the wedge product is taken in the order induced on $I \sqcup \wt{I}$ from the total order $1 < \wt{1} < 2 < \wt{2} < \cdots < n < \wt{n}$. The $e_{I,\wt I}$ form a basis for $\extp^{n+1} \R^{2n} \cong \R^{{2n \choose n+1}}$. Let $\Delta_{I,\wt I}$ denote the coordinate of $\extp^{n+1} \R^{2n}$ corresponding to $e_{I,\wt I}$. We put
	\[ \Delta_{I, \tilde{I}} = \sum_{(\sigma,\tilde{\sigma}) \ \text{concordant with} \ (I, \tilde{I})} \Lambda_{\sigma} . \]
	So the $\Delta$'s are related to the $\Lambda$'s by a linear map, which can be checked~\cite{Lam}*{Proposition~5.19} to be injective, and we get a linear embedding of $\R \P^{\Cat_n -1}$ into $\R \P^{\binom{2n}{n+1} -1}$.  {We refer to the subset of $\R \P^{\binom{2n}{n+1} -1}$ where the coordinates $\Delta_{I, \tilde{I}}$ are realized by some cactus network as $\mathcal{X}_n$.}
	\begin{eg} \label{DeltaEG}
		Continuing with our running example, we have
		\[ \Delta_{\{1\},\, \{\wt{1} \ \wt{2}\ \wt{3} \}} =  \Delta_{\{2\},\, \{\wt{1} \ \wt{2}\ \wt{3} \}} =  \Delta_{\{3\},\, \{\wt{1} \ \wt{2}\ \wt{3} \}} = \Lambda_{\{123\}}= abc \]
		\[ \Delta_{\{12 \},\, \{\wt{1}\ \wt{2}\}} = \Delta_{\{ 13 \},\, \{\wt{2}\ \wt{3}\}} = \Lambda_{\{1\}, \{23 \}} = bc \ \text{and rotations thereof}\]
		\[ \Delta_{\{12 \},\ \{\wt{2}\ \wt{3}\}} = \Lambda_{\{1 \},\, \{23 \}} + \Lambda_{\{2\},\ \{13 \}} = bc+ac \ \text{and rotations thereof} \]
		\[ \Delta_{\{ 123 \},\, \{ \wt{1} \}} = \Delta_{\{ 123 \},\, \{ \wt{2} \}} = \Delta_{\{ 123 \},\, \{ \wt{3} \}} = \Lambda_{\{1\},\, \{2\},\, \{3 \}} = a+b+c. \]
		The phrase ``and rotations thereof" indicates that one can, in each, case, obtain two more relations similar to this one by rotating the subsets of $[3] \sqcup [\wt{3}]$ being considered.
	\end{eg}

	\begin{remark}\label{rem:complement_warning}
		Lam actually works with the complementary set, $[n] \sqcup [\tilde{n}] \setminus (I \sqcup \tilde{I})$, so he refers to $\binom{2n}{n-1}$ throughout. Lam would say that $I$ and $\sigma$ are concordant when we would say that $[n] \setminus I$ and $\sigma$ are. We find our convention more convenient and will silently convert all of our references to Lam to use this complementary notation.
	\end{remark}
	
	We can rewrite Equations~\eqref{LambdaRatios} in terms of the $\Delta$'s in several equivalent ways, such as the one below. For convenience, we take $1 \leq i < j \leq n$:
	\begin{equation} L_{ij} = \frac{\Delta_{[n]\setminus \{ j \},\ \{\wt{i-1}, \wt{i} \}}}{\Delta_{[n],\ \{ \wt{i} \}}}
		\qquad 
		R_{ij} = \frac{\sum_{k=i}^{j-1} \Delta_{\{i,j \},\ [\wt{n}] \setminus \{ \wt{k} \}}}{\Delta_{\{i \},\ [\wt{n}]}} .
		\label{DeltaRatios} 
	\end{equation}
	
	{Lam shows that every point in $\mathcal{X}_n$ is the image, under the \emph{Pl\"ucker embedding} of a point in the \emph{Grassmannian}. }
	The {Grassmannian} $\Gr(n+1,2n)$ is the space of $(n+1)$-dimensional subspaces of $\R^{2n}$.  We can represent a point $X \in\Gr(n+1,2n)$ as an $(n+1)\times 2n$ matrix $M$ with rows $v_1,\dots,v_{n+1}$ such that $X=\Span( v_1,\dots,v_{n+1} )$.  Note that the matrices that satisfy these conditions are exactly those matrices we can obtain from $M$ by row operations. This lets us identify ${\rm Gr}(n+1,2n)$ with full rank $(n+1)\times 2n$ matrices modulo row operations. The {Pl\"ucker embedding} is the map
	\begin{align*}
		\iota:\Gr(n+1,2n) &\hookrightarrow \P{(\extp^{n+1} \R^{2n}})\\
		X=\Span( v_1,\dots,v_{n+1} )  &\mapsto [v_1 \wedge \dots \wedge v_{n+1}],
	\end{align*}
	where $v_1,\dots,v_{n+1}$ are vectors that span $X$.
	In coordinates, $\Delta_{I,\wt I}(X):=|M_{i_1,\dots,i_k}|$ where $|M_{i_1,\dots,i_k}|$ is the maximal minor obtained from columns $I \sqcup \wt I=\{i_1<i_2<\dots< i_{n+1}\}$ of $M$.  These coordinates on ${\rm Gr}(n+1,2n)$ are called \emph{Pl\"ucker coordinates}.  Note that this map is well defined because row operations scale all maximal minors of $M$ by a constant.	The \emph{totally nonnegative Grassmannian}, ${\rm Gr}_{\geq 0}(n+1, 2n)$, introduced by Postnikov~\cite{Pos}, is the subset of ${\rm Gr}(n+1, 2n)$ where all Pl\"ucker coordinates are nonnegative. Lam's main theorem (stated as Theorem \ref{theorem:Lammain} below) is that {$\mathcal{X}_n$} is the intersection, inside $\R\P^{\binom{2n}{n+1}-1}$, of ${\rm Gr}_{\geq 0}(n+1, 2n)$ with $\R\P^{\Cat_n -1}$.

	\begin{eg} \label{RowSpanEG}
		Continuing with our running example, the row span of the matrix 
		\[ X=
		\begin{bmatrix}
			0 & a+b+c & 0 & -a-b-c & 0 & a+b+c \\
			1 & \frac{a b}{a+b+c} & 0 & 0 & 0 & -\frac{a c}{a+b+c} \\
			0 & -\frac{a b}{a+b+c} & -1 & -\frac{b c}{a+b+c} & 0 & 0 \\
			0 & 0 & 0 & \frac{b c}{a+b+c} & 1 & \frac{a c}{a+b+c} \\
		\end{bmatrix} \]
		has Pl\"ucker coordinates as given in Example~\ref{DeltaEG}. Note that the columns appear in order $1$, $\wt{1}$, $2$, $\wt{2}$, $3$, $\wt{3}$. The first row, rescaled by $1/(a+b+c)$ is $(0,1,0,-1,0,1)$ and the sum of the other three rows is $(1,0,-1,0,1,0)$.
		%
		
		There are, of course, many matrices with the same row span. Another matrix whose Pl\"ucker coordinates are the same up to a global sign, and which hence has the same row span, is.
		\[ \wt{X} =\begin{bmatrix}
			a b c & 0 & -a b c & 0 & a b c & 0 \\
			\frac{1}{a} & 1 & \frac{1}{b} & 0 & 0 & 0 \\
			0 & 0 & -\frac{1}{b} & -1 & -\frac{1}{c} & 0 \\
			-\frac{1}{a} & 0 & 0 & 0 & \frac{1}{c} & 1 \\
		\end{bmatrix} .\]
	\end{eg}
	
	Lam's paper leaves it as a mystery how one should understand this linear slice of ${\rm Gr}(n+1, 2n)$.
	The goal of this paper is to answer that mystery.
	Let $\Omega$ be the skew-symmetric bilinear form 
	\begin{multline} \Omega((x_1, x_{\wt{1}}, x_2, x_{\wt{2}}, \dots, x_n, x_{\wt{n}}),\ (y_1, y_{\wt{1}}, y_2, y_{\wt{2}}, \dots, y_n, y_{\wt{n}})) = \\
		\sum_{i=1}^n (x_i y_{\wt{i}} - x_{\wt{i}} y_i) + \sum_{j=1}^{n-1} (x_{j+1} y_{\wt{j}} - x_{\wt{j}} y_{j+1}) + (-1)^n (x_1 y_{\wt{n}} - x_{\wt n} y_1) .  \label{OmegaFormula} \end{multline}
	
	The form $\Omega$ has a two dimensional kernel, spanned by the vectors $(0,1,0,-1,\dots,0,(-1)^{n-1})$ and $(1,0,-1,0,\dots,(-1)^{n-1},0)$. We define an $(n+1)$-dimensional subspace $X$ of $\R^{2n}$ to be \text{isotropic} for $\Omega$ if $\Omega(\vec{x}, \vec{y})=0$ for any $\vec{x}$ and $\vec{y}$ in $X$. 
	We define $\IG^{\Omega}(n+1, 2n)$ to be the space of isotropic subspaces inside ${\rm Gr}(n+1, 2n)$. 
	
	Since $\Omega$ has corank $2$, $\IG^{\Omega}(n+1, 2n)$ is isomorphic as an abstract variety to the Lagrangian grassmannian $\LG(n-1,2n-2)$. 
	However, as we will discuss in Section~\ref{sec:Karpman}, the total positivity structure on $\IG^{\Omega}(n+1,2n)$ is rather different from those previously studied for Lagrangian Grassmannians.
	
	Then our main result is 
	
	\begin{theorem} \label{thm:main}
		The intersection $\R\P^{\Cat_n -1} \cap {\rm Gr}(n+1, 2n)$ inside $\R\P^{\binom{2n}{n+1} -1}$ is $\IG^{\Omega}(n+1,2n)$.
		As a consequence, {$\mathcal{X}_n$} is the space of $(n+1)$-dimensional subspaces of $\R^{2n}$ which are both totally nonnegative and are isotropic for the form $\Omega$.
	\end{theorem}
	
	We pause to acknowledge that the signs in Equation~\eqref{OmegaFormula} are annoying. 
	They are necessary if we want our isotropic subspaces to have nonnegative Pl\"ucker coordinates.
	If we are willing to sacrifice this, there is a simpler sign choice. Let $D$ be the $(2n) \times (2n)$ diagonal matrix with $D_{ii} = D_{\wt{i}\ \wt{i}} = (-1)^{i-1}$. 
	Conjugating $\Omega$ by $D$ produces the simpler form
	\begin{equation} \Omega^D((x_1, x_{\wt{1}}, \dots, x_n, x_{\wt{n}}),\ (y_1, y_{\wt{1}},\dots, y_n, y_{\wt{n}})) =
		\sum_{i=1}^n (x_i y_{\wt{i}} - x_{\wt{i}} y_i) + \sum_{j=1}^{n} (x_{\wt{j}} y_{j+1} - x_{j+1} y_{\wt{j}}) .  \label{OmegaDFormula} \end{equation}
	with indices cyclic modulo $n$. The kernel of $\Omega^D$ is spanned by $(1,0,1,0,\cdots,1,0)$ and $(0,1,0,1,\cdots,0,1)$; multiplication by $D$ carries $\IG^{\Omega}(n+1,2n)$ to $\IG^{\Omega^D}(n+1,2n)$. 
	The reader may like to check that the rows of the matrices $X$ and $\wt{X}$ from Example~\ref{RowSpanEG} are isotropic for $\Omega$. 
	The reader may also enjoy computing the prettier, but not totally positive, matrices $XD$ and $\wt{X} D$ and checking that their rows are isotropic for $\Omega^D$.

	We now describe coordinates for an affine patch in $\IG^{\Omega}(n+1, 2n)$, and then describe how they can be used to give explicit formulas for our isotropic subspace in terms of the $L$ matrix and the $R$ matrix. As will be a trend, the formulas are slightly nicer for $\Omega^D$ than for $\Omega$.

	Let $V$ be an isotropic $n+1$ plane for $\Omega$. We will show (Lemma~\ref{lem:extreme_Deltas}) that $\Delta_{[n], \{ \wt{1} \}}(V) = \Delta_{[n], \{ \wt{2} \}}(V) = \cdots = \Delta_{[n], \{ \wt{n} \}}(V)$ and $\Delta_{\{ 1 \}, [\wt{n}]}(V) = \Delta_{\{ 2 \}, [\wt{n}]}(V) = \cdots = \Delta_{\{ n \}, [\wt{n}]}(V)$. Let $U_{\text{not shorted}}$ be the open set where the $\Delta_{[n],\{\wt{k}\}}$ are nonzero and $U_{\text{connected}}$ be the open set where the $\Delta_{\{ 1 \}, [\wt{n}]}$ are nonzero. The totally nonnegative points of $U_{\text{not shorted}}$ correspond to planar electrical networks; the totally nonnegative points of $U_{\text{connected}}$ correspond to cactus networks where the underlying electrical network is connected; see Lemma \ref{lem:conshort}.
	
	The spaces $U_{\text{not shorted}}$ and $U_{\text{connected}}$ are Schubert cells of maximal dimension in $\IG^{\Omega}(n+1, 2n) \cong \LG(n-1,2n-2)$, and hence isomorphic to $\R^{\binom{n}{2}}$. We now 
	describe explicit isomorphisms between these spaces and the space of symmetric $n \times n$ matrices with row and column sum $0$.
	
	\begin{theorem} \label{thm:SchubertChart}
		Every subspace in $U_{\text{not shorted}}$ can be written as the row span of a matrix of the form
		\[ 
		\begin{bmatrix}
			0&1&0&1&\cdots&0&1 \\
			1&S_{11}&0&S_{12}&\cdots&0&S_{1n} \\
			0&S_{21}&1&S_{22}&\cdots&0&S_{2n} \\
			\vdots&\vdots&\vdots&\cdots&\vdots&\vdots\\
			0&S_{n1}&0&S_{n2}&\cdots&1&S_{nn} \\
		\end{bmatrix} D. \]
		The matrix $S$ is unique up to adding a multiple of $(1,1,\ldots,1)$ to each row, and the matrix
		\[ \begin{bmatrix}
			S_{1n}-S_{11} & S_{2n}-S_{21} & \cdots & S_{nn}-S_{n1} \\
			S_{11}-S_{12} & S_{21}-S_{22} & \cdots & S_{n1}-S_{n2} \\
			\vdots & \vdots & \cdots & \vdots \\
			S_{1(n-1)} - S_{1n} & S_{2(n-1)} - S_{2n}  & \cdots & S_{n(n-1)} - S_{nn} \\
		\end{bmatrix} \]
		is a symmetric matrix whose rows and columns add to $0$. Conversely, given a symmetric matrix whose rows and columns add to $0$, we get a unique point of $U_{\text{not shorted}}$ in this manner.
		
		Similarly, subspaces in $U_{\text{connected}}$ are expressible as the row span of a matrix of the form
		\be \label{mat:t}
		\begin{bmatrix} 
			1&0&1&0&\cdots&1&0 \\
			T_{11}&1&T_{12}&0&\cdots&T_{1n}&0 \\
			T_{21}&0&T_{22}&1&\cdots&T_{2n}&0 \\
			\vdots&\vdots&\vdots&\vdots&\cdots&\vdots&\vdots\\
			T_{n1} & 0&T_{n2}&0 & \cdots & T_{nn}&1 \\
		\end{bmatrix} D \ee
		where
		\[ \begin{bmatrix}
			T_{12}-T_{11}&T_{22}-T_{21} & \cdots & T_{n2}-T_{n1} \\
			T_{13}-T_{12}&T_{23}-T_{22} & \cdots & T_{n3}-T_{n2} \\
			\vdots & \vdots & \cdots & \vdots \\
			T_{11}-T_{1n}&T_{21}-T_{2n} & \cdots & T_{n1}-T_{nn} \\
		\end{bmatrix} 
		\]
		is symmetric with rows and columns adding to $0$.
	\end{theorem}
	
	\begin{eg} \label{eg:sym_matrices}
		In our running example,
		\[  XD=
		\begin{bmatrix}
			0 & a+b+c & 0 & a+b+c & 0 & a+b+c \\
			1 & \frac{a b}{a+b+c} & 0 & 0 & 0 & -\frac{a c}{a+b+c} \\
			0 & -\frac{a b}{a+b+c} & 1 & \frac{b c}{a+b+c} & 0 & 0 \\
			0 & 0 & 0 & -\frac{b c}{a+b+c} & 1 & \frac{a c}{a+b+c} \\
		\end{bmatrix} .\]
		Rescaling the top row to $(1,0,1,0,1,0)$ doesn't change the span, so the matrix $S$ is $ \tfrac{1}{a+b+c} \begin{sbm} ab&0&-ac \\ -ab&bc&0 \\ 0&-bc&ac \\ \end{sbm}$. The corresponding symmetric matrix is \[\frac{1}{a+b+c} \begin{bmatrix} -ab-ac&ab&ac \\ ab&-ab-bc&bc \\ ac&bc&-ac-bc \\ \end{bmatrix}.\] 
		
		Similarly,
		\[ \wt{X} D =\begin{bmatrix}
			a b c & 0 & a b c & 0 & a b c & 0 \\
			\frac{1}{a} & 1 & -\frac{1}{b} & 0 & 0 & 0 \\
			0 & 0 & \frac{1}{b} & 1 & -\frac{1}{c} & 0 \\
			-\frac{1}{a} & 0 & 0 & 0 & \frac{1}{c} & 1 \\
		\end{bmatrix} .\]
		So the matrix $T$ is $\begin{sbm} \tfrac{1}{a} & - \tfrac{1}{b} & 0 \\ 0 & \tfrac{1}{b} & -\tfrac{1}{c} \\ -\tfrac{1}{a} & 0 & \tfrac{1}{c} \end{sbm}$ and the corresponding symmetric matrix is
		\[ \begin{bmatrix}
			-\frac{1}{a}-\frac{1}{b} & \frac{1}{b} & \frac{1}{a} \\[0.1 cm]
			\frac{1}{b} & -\frac{1}{b}-\frac{1}{c} & \frac{1}{c} \\[0.1 cm]
			\frac{1}{a} & \frac{1}{c} & - \frac{1}{a}-\frac{1}{c} \\[0.1 cm]
		\end{bmatrix} .\]
	\end{eg}
	
	Using these isomorphisms, we can describe explicitly how to turn response matrices and effective resistances into points of the Grassmannian.

	\begin{theorem} \label{thm:Lmatrix_chart}
		Let $G$ be a planar electrical network with response matrix $L$. The isotropic plane in $U_{\text{not shorted}}$ corresponding to $G$ corresponds to the symmetric matrix $L$.
	\end{theorem}
	
	For example, the first displayed symmetric matrix in Example~\ref{eg:sym_matrices} is our response matrix $L$.
	
	\begin{theorem}  \label{thm:Rmatrix_chart}
		Let $G$ be a connected cactus network with effective resistance matrix $R$. Set $L^*_{ij} = \tfrac{1}{2} (R_{ij} + R_{i+1,j+1} - R_{i+1,j} - R_{i,j+1} )$ (indices are periodic modulo $n$). 
		The isotropic plane in $U_{\text{connected}}$ corresponding to $G$ corresponds to the symmetric matrix $L^*$.
	\end{theorem}
	
	For example, the second displayed symmetric matrix in Example~\ref{eg:sym_matrices} is ${L^*}$ with respect to our effective resistance matrix $R$.

	The authors' original goal was to understand the relation between the appearance of ${\rm Gr}(n-1, 2n)$ in Lam's work and the appearance of the orthogonal Grassmannian ${\rm OG}(n-1, 2n)$ in the work of Henriques and the third author~\cite{HenriquesSpeyer}. 
	We have not yet realized this goal, but we believe there are enough ideas in this paper to be worth recording; we hope to return to our original goal in future work.
	
	\subsection{Acknowledgments} The third author first discussed these ideas with Pavel Galashin at FPSAC 2018, and is grateful to Pavel for this and many other conversations. 
	This material is based upon work supported by the National Science Foundation under Grant No. DMS-1439786 while the first author was in residence at the Institute for Computational and Experimental Research in Mathematics in Providence, RI, during the Spring 2021 semester.
	The third author was supported by the National Science Foundation under Grant No. DMS-1855135 and DMS-1854225
	
	\section{Background}

	\subsection{Cactus networks}\label{sec:cactus-nets}
	
	Let $\D$ be a disk with $n$ points labeled $1,2,\dots,n$ in clockwise order around the boundary and let $ \sigma $ be a non-crossing partition of $[n]$. A \textit{cactus with shape $\sigma$} is the topological space $\D/\sigma$ obtained by gluing the points $i$ that are in the same block of $\sigma$. $\D/\sigma$ consists of $|\tilde \sigma|$ disks glued together. A \textit{cactus network with shape $\sigma$} is a graph $G$ embedded in $\D/\sigma$ with  boundary vertices $[n]$, along with a function $c:E \ra \R_{>0}$ called \textit{conductance}. A cactus network decomposes into a union of $|\tilde \sigma|$ planar electrical networks. 
	If $\sigma$ is the partition $\{1\},\{2\},\dots,\{n\}$ consisting of all singletons, then $\D/\sigma=\D$ and $G$ is a planar electrical network. These networks have been studied extensively, see e.g.~\cites{CIM, VGV, KW11}.  
	\begin{figure}
		\centering
		{\includegraphics[width=0.65\textwidth]{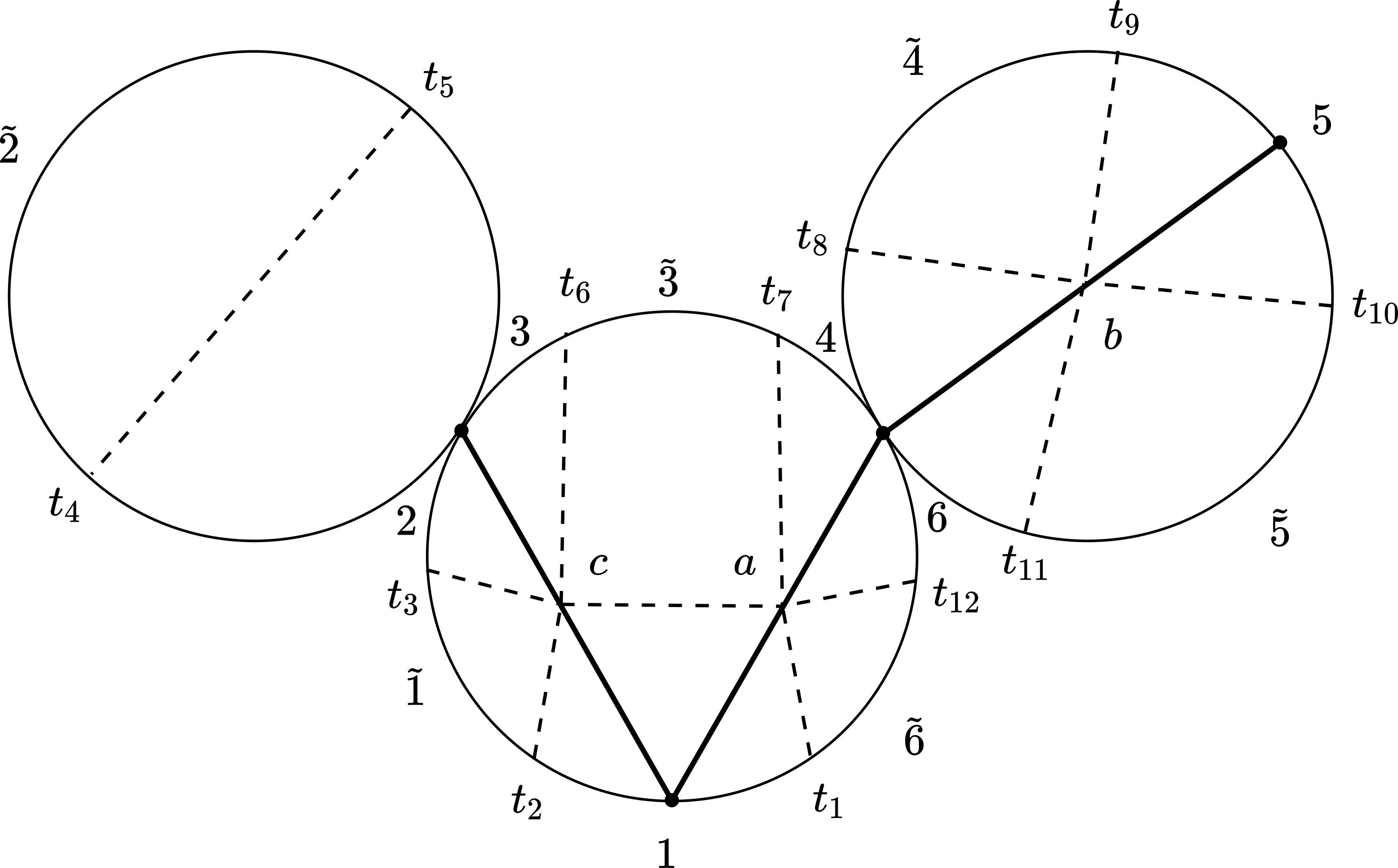}}
		\caption{A cactus network $G$ (solid) and its medial graph (dashed).}\label{fig:cactus}
	\end{figure}
	Figure \ref{fig:cactus} shows a cactus network with shape $\{1\}, \{2,3\},\{4,6\},\{5\}$. It decomposes into a union of $3$ planar electrical networks, corresponding to the blocks $\{\tilde 1,\tilde 3,\tilde 6 \},\{\tilde 2 \},\{\tilde 4, \tilde 5 \}$ of $\tilde \sigma$.

	The medial graph $G^\times$ of a cactus network $G$ is defined as follows. Place vertices $t_1,\dots,t_{2n}$ clockwise around the boundary of the cactus such that $i$ is between $t_{2i-1}$ and $t_{2i}$, and a vertex $t_e$ at the midpoint of each edge $e$ of $G$. For $e,e' \in E$, draw an edge between $t_e$ and $t_{e'}$ in $G^\times$ if there is a face of $G$ around which $e $ and $e'$ occur consecutively.
	Draw an edge from $t_{2i-1}$ (respectively $t_{2i}$) to $t_e$ if $e$ is the first (respectively last) edge in clockwise order around the boundary face of $G$ containing $t_{2i-1}$ (respectively $t_{2i}$).
	If $i$ is an isolated vertex, draw an edge between $t_{2i-1}$ and $t_{2i}$.
	
	Note that each boundary vertex has degree $1$ and each interior vertex $t_e$ has degree $4$. A \textit{medial strand} in $G^\times$ is a path in $G^\times$ that starts at a boundary vertex $t_i$, follows the only edge incident to it, and then at each interior vertex of degree $4$, follows the edge opposite to the one used to arrive at it. There are $n$ medial strands, and they give rise to a matching $\tau(G)$ on $[2n]$:
	$$
	\tau(G):=\{\{i,j\} \subset [2n]: \text{there is a medial strand in $G^\times$ from $t_i$ to $t_j$} \}.
	$$
	The matching $\tau(G)$ is called the \textit{medial pairing associated with} $G$. Let $P_n$ denote the set of matchings on $[2n]$.
	
	\begin{remark}
		Every $\tau\in P_n$ is the medial pairing for some cactus network $G$.  Some of these matchings are not medial pairings for  planar networks.  For example $\{4,5\}$ could not be a part of $\tau(G)$ for any planar network $G$, but we see in Figure~\ref{fig:cactus} that this can happen in a cactus network.  Thus, cactus networks can be thought of as a natural generalization of circular planar networks where all medial pairings are allowed.
	\end{remark}
	
	A graph $G$ embedded in a cactus is said to be \textit{minimal} if the medial strands in $G^\times$ have no self intersections and there are no bigons (two medial strands that intersect each other twice).
	
	\begin{example}
		The cactus network in Figure \ref{fig:cactus}  is minimal and has medial pairing
		$$
		\{\{1,7\},\{2,6\},\{3,{12}\},\{4,5\},\{8,{10}\},\{9,{11}\}\}.
		$$
	\end{example}
	
	\subsection{The space of cactus networks.}

	\begin{figure}
		\centering
		{\includegraphics[width=0.5\textwidth]{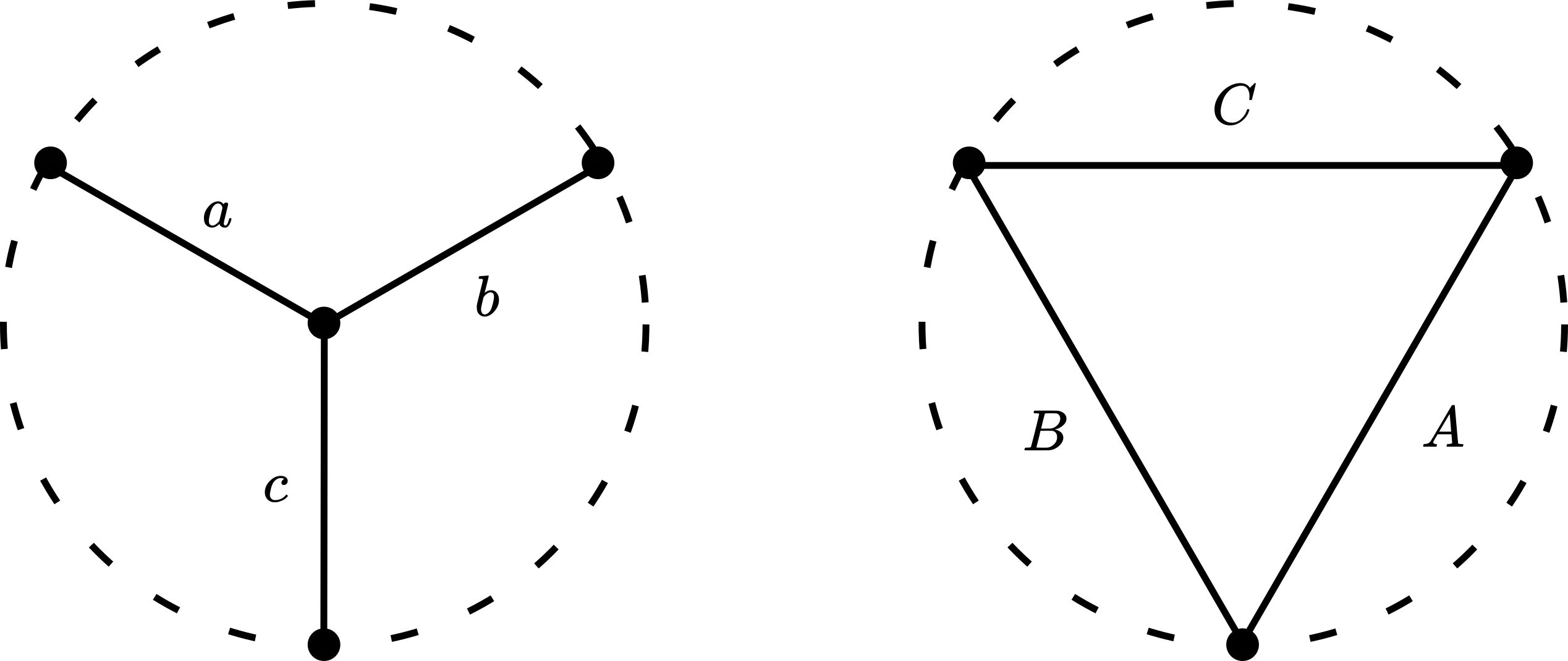}}
		\caption{The Y-$\Delta$ move.}\label{fig::ydelta}
	\end{figure}

	There is a local move on cactus networks called the Y-$\Delta$ move (see Figure \ref{fig::ydelta}). Two graphs $G_1$ and $G_2$ embedded in a cactus are \textit{topologically equivalent} if there is a sequence of Y-$\Delta$ moves such that $G_1 \mapsto G_2$. 
	\begin{proposition}[\cite{Lam18}*{Proposition 4.2}]\la{lam1}
		The function assigning to each cactus network with boundary vertices $[n]$ its medial pairing gives a bijection
		$$
		\{\text{Minimal cactus graphs}\}/\text{topological equivalence} \xrightarrow[]{\sim} P_n.
		$$
	\end{proposition}
	
	Let $\tau \in P_n$. For a minimal graph $G$ embedded in a cactus with $\tau(G)=\tau$, let 
	$$
	\mathcal R_G:=\{c:E \ra \R_{>0}\},
	$$
	be the space of minimal cactus networks with underlying graph $G$. A Y-$\Delta$ move $G_1 \mapsto G_2$ induces a homeomorphism $\mathcal R_{G_1} \ra \mathcal R_{G_2}$ given in the notation of Figure \ref{fig::ydelta} by
	$$
	A=\frac{bc}{a+b+c}, \quad B=\frac{ac}{a+b+c}, \quad C=\frac{ab}{a+b+c}.
	$$
	
	Two minimal cactus networks $(G_1,c_1)$ and $(G_2,c_2)$ are said to be \textit{electrically equivalent} if there is a sequence of Y-$\Delta$ moves $(G_1,c_1) \mapsto (G_2,c_2)$. Gluing the $\mathcal R_G$ for all minimal $G$ with $\tau(G)=\tau$ using the bijections induced by Y-$\Delta$ moves, we obtain the space $\mathcal R_\tau$ parameterizing electrical equivalence classes of minimal cactus networks with medial pairing $\tau$. Let $\mathcal R_n$ denote the space of electrical equivalence classes of minimal cactus networks with boundary vertices $[n]$. By Proposition \ref{lam1}, $\mathcal R_n$ has the stratification 
	$$
	\mathcal R_n=\bigsqcup_{\tau \in P_n}\mathcal R_\tau.
	$$
	
	Lam \cite{Lam18} uses the grove measurements $\Lambda_\sigma$ to identify $\mathcal R_n$ with a closed (in the Euclidean topology) subset of $\R \P^{\Cat_n-1}$. We use the induced topology to make $\mathcal R_n$ a topological space 
	{, which we call the \emph{space of cactus networks}}.

	\subsection{Response and effective resistance matrices}\label{sec:harm-iso}
	Suppose $G$ is a cactus network with shape $\sigma$. {Let $\Gamma$ be obtained from $G$ be relabeling the boundary vertices by blocks of $\sigma$.  This means that vertices identified by $\sigma$ have only one label.} 
	We call the vertices of $\Gamma$ corresponding to blocks of $\sigma$ \textit{boundary vertices}. The \textit{Laplacian} on $\Gamma$ is the linear operator
	$
	{{\mathcal{L}}}:\R^{\text{vertices of }\Gamma} \ra \R^{\text{vertices of }\Gamma}
	$
	defined by 
	$$
	({\mathcal{L}} f)(u) := \sum_{\text{edges } \{u,v\}  } c(\{u,v\}) (f(u)-f(v))
	$$
	where the sum is over all vertices $u$ that are incident to $v$. In the standard basis of $\R^{\text{vertices of }\Gamma}$, ${\mathcal{L}}$ is represented by the symmetric matrix
	$$
	{\mathcal{L}}_{u,v}=\begin{cases}
		\sum_{\text{edges }\{v,v'\} } c(\{v,v'\}) &\text{ if } u=v,\\
		-c(\{u,v\}) &\text{ if $u \neq v$ and $\{u,v\} \text{ is an edge of }\Gamma $},\\
		0 &\text{ if $u \neq v$ and $\{u,v\}\text{ is not an edge of }\Gamma $}.
	\end{cases}
	$$
	
	A function $f$ on the vertices of $\Gamma$ is called a \textit{harmonic} function if $(\mathcal{L} f) (u)=0$ for all non-boundary vertices $u$ of $\Gamma$. Given a function $F$ on the boundary vertices of $\Gamma$, the \textit{Dirichlet problem} asks for a harmonic function $f$ that agrees with $F$ on the boundary vertices. The Dirichlet problem has a unique solution and the function $f$ is called the \textit{harmonic extension} of $F$.
	
	Now we can define the electrical terminology used in the introduction. A \textit{voltage} $V$ on $G$ is a harmonic function. The \textit{current} $J$ associated to $V$ is a function on the directed edges of $\Gamma$ defined by $J(u,v):=c(\{u,v\})(V(v)-V(u))$, where $(u,v)$ denotes an edge of $\Gamma$ directed from vertex $u$ to vertex $v$. Note that $J$ is antisymmetric: $J(v,u)=-J(u,v)$. The quantity $(\mathcal{L} V)(u)$ is the net current flowing into the vertex $u$ when the vertices of $\Gamma$ are held at voltages given by $V$. For a boundary vertex $v$ of $\Gamma$, let $J_u:=\sum_{ v \text{ incident to }u}J(u,v)$ denote the total current flowing out vertex $u$. Define $L:\R^{\text{boundary vertices of }\Gamma} \ra \R^{\text{boundary vertices of }\Gamma}$ to be the map that sends a vector $(V_u)$ of voltages of boundary vertices to the vector $(J_u)$ of currents flowing out of each boundary vertex when the vertices are held at the voltages determined by the harmonic extension $V$ of $(V_u)$.  This map is linear and the matrix $L$ is called the \textit{response matrix} of $\Gamma$. The response matrix is symmetric, has rows and columns that add up to zero, and can be explicitly constructed as the Schur complement of the Laplacian with respect to the square submatrix of $\mathcal{L}$ corresponding to the non-boundary vertices of $\Gamma$ multiplied by $-1$.
	
	We now define the \textit{effective resistance matrix} $R$ of a connected cactus network $G$. For boundary vertices $i,j$ of $G$, let $u,v$ denote the corresponding vertices of $\Gamma$.  
	Let $V$ denote a solution to $LV=e_u-e_v$, where $e_u,e_v$ are standard basis vectors of $\R^{\text{boundary vertices of }\Gamma}$. Note that although $L$ is not invertible, since $\Gamma$ is connected, the cokernel of $L$ is spanned by the vector $(1,1,\dots,1)$ and therefore $e_v-e_u \in (1,1,\dots,1)^\perp$ is in the image of $L$. Define $R_{ij}:=V(v)-V(u)$, the voltage difference between $v$ and $u$ so that one unit of current flows from $u$ to $v$. Although  $V$ is only defined modulo an additive constant, $R_{ij}$ is well defined. Notice that $R$ is a symmetric matrix with zeroes on the diagonal.
	
	{Kirchhoff's formulas (\ref{ResponseFormula}) and (\ref{ResistanceFormula}) express the matrices $L$ and $R$ in terms of ratios of grove measurements.
	}
	\begin{eg}
		For the cactus network $G$ in Figure \ref{fig:cactus}, $\Gamma$ is a planar electrical network with four boundary vertices $\{1\},\{2,3\},\{4,6\},\{5\}$ and no non-boundary vertices. We have
		$$
		-\mathcal{L}=L=\begin{bmatrix}
			-a-c & c& a& 0\\
			c & -c &  0&0\\
			a&0&-a-b&b\\
			0&0&b&-b
		\end{bmatrix}.
		$$
		Let us compute $R_{25}$. We have $u=\{2,3\},v=\{5\}$. The voltage $V=(0,-\frac{1} c,\frac 1 a,\frac {a+b}{ab})$ satisfies $LV=e_v-e_u$, so $R_{25}=\frac 1 a+\frac 1 b+\frac 1 c.$ We have $\Lambda_{\{1,2,3,4,5,6\}}=abc$ and there are three non-crossing partitions concordant with $\{2,5\}$ that have nonzero contribution to the sum in (\ref{LambdaRatios}):
		$$
		\Lambda_{\{2,3\},\{1,4,5,6\}}=ab, \quad \Lambda_{\{1,2,3\},\{4,5,6\}}=bc, \quad \Lambda_{\{1,2,3,4,6\},\{5\}}=ac.
		$$
		Therefore we get $$
		\frac{\Lambda_{\{2,3\},\{1,4,5,6\}}+ \Lambda_{\{1,2,3\},\{4,5,6\}}+ \Lambda_{\{1,2,3,4,6\},\{5\}}}{\Lambda_{\{1,2,3,4,5,6\}}}=\frac{ab+bc+ac}{abc}=\frac 1 a+\frac 1 b+\frac 1 c=R_{25},
		$$
		verifying (\ref{LambdaRatios}).
	\end{eg}

	\subsection{Lam's map $\mathcal{T}$}\label{sec:Lam-map}
	
	Let $V$ be the vector space with basis $e_1, e_{\wt{1}},\dots,e_n,e_{\wt{n}}$. For $I \sqcup \wt{I} \subseteq [n] \sqcup [\wt{n}]$, we put $e_{I, \wt{I}} = \Alt_{i \in I \sqcup \wt{I}} e_i$, where the wedge product is taken in the order induced on $I \sqcup \wt{I}$ from the total order $1 < \wt{1} < 2 < \wt{2} < \cdots < n < \wt{n}$. Given a Kreweras pair  $(\sigma, \wt{\sigma})$, let 
	\begin{equation} f_{\sigma} = \sum_{(I, \wt{I}) \ \text{concordant with}\ (\sigma, \wt{\sigma})} e_{I, \wt{I}} . \label{eq:f_sigma_defn} \end{equation}
	Lam~\cite{Lam18} defines a map $\mathcal{T}:\R^{\Cat_n}\to\Alt^{n+1} V$ by $(\Lambda_{\sigma})\mapsto\sum_{\sigma} \Lambda_{\sigma} f_{\sigma}$.
	Since $V\cong\R^{2n}$, we can identify $\P(\Alt^{n+1} V)$ and $\R\P^{\binom{2n}{n+1}-1}$ using the basis $e_{I,\wt I}$. From here on we will use these spaces interchangeably.
	
	It is easy to check that $\mathcal{T}$ is injective \cite{Lam18}*{Proposition 5.19}. So the projectivization of the image of $\mathcal{T}$, or $\P(\mathcal{T}(\R^{\Cat_n}))$ is a linear subspace isomorphic to $\R\P^{\Cat_n-1}$.  {We call this subspace $\mathcal{H}_n$.}
	\begin{theorem}[\cite{Lam18}] \label{theorem:Lammain}
		Lam's map induces a map from $\mathcal R_n$ to $\P(\Alt^{n+1} V)$ taking a cactus network with grove measurements $\Lambda_\sigma$ to $[\sum_{\sigma} \Lambda_{\sigma} f_{\sigma}]$. This map is a homeomorphism of $\mathcal R_n$ with ${\mathcal{X}_n:=\mathcal{H}_n}\cap \Gr_{\geq 0}(n+1,2n) {\subseteq}\R\P^{\binom{2n}{n+1}-1}$.
	\end{theorem}
	
	{Note that Lam calls $\mathcal{X}_n$ the space of cactus networks.  We reserve this name for $\mathcal{R}_n$, even though the spaces are homeomorphic.}

	\subsection{Duality and cyclic symmetry}\label{sec:dual}
	The dual cactus network $ G^*$ associated to $G$ is defined as follows. Since the medial graph $G^\times$ is disjoint from $[n]$, we can identify it with its preimage in $\D$ under the quotient map $\D \ra \D/\sigma$. $\D \setminus G^\times $ consists of two types of regions corresponding to vertices of $G$ and faces of $G$. We place a vertex of $G^*$ in every region corresponding to a face. For every vertex $t_e$ of $G^\times$, we connect the two faces incident to $e$ by a dual edge $e^*$ in $G^*$ and assign it conductance $c^*(e^*)=\frac{1}{c(e)}$. We get a partition $\sigma^*$ whose parts consist of $\wt i$ that are incident to the same face of $G^*$. Passing to the quotient $\D/\sigma^*$, we get the dual cactus network $G^*$ embedded in $\D/\sigma^*$ with boundary vertices $[\wt n]$. Let $s$ denote the operation of cyclically shifting the labels of the boundary points $1<\wt 1 < 2 < \wt 2 \dots < n < \wt n$ clockwise by one step (so that $s(1)=\wt n, s(\wt 1)=1$ etc). Applying $s$, $G^*$ becomes a cactus network in $\D/s(\sigma^*)$ with boundary vertices $[n]$.
	\begin{figure}
		\centering
		{\includegraphics[width=0.25\textwidth]{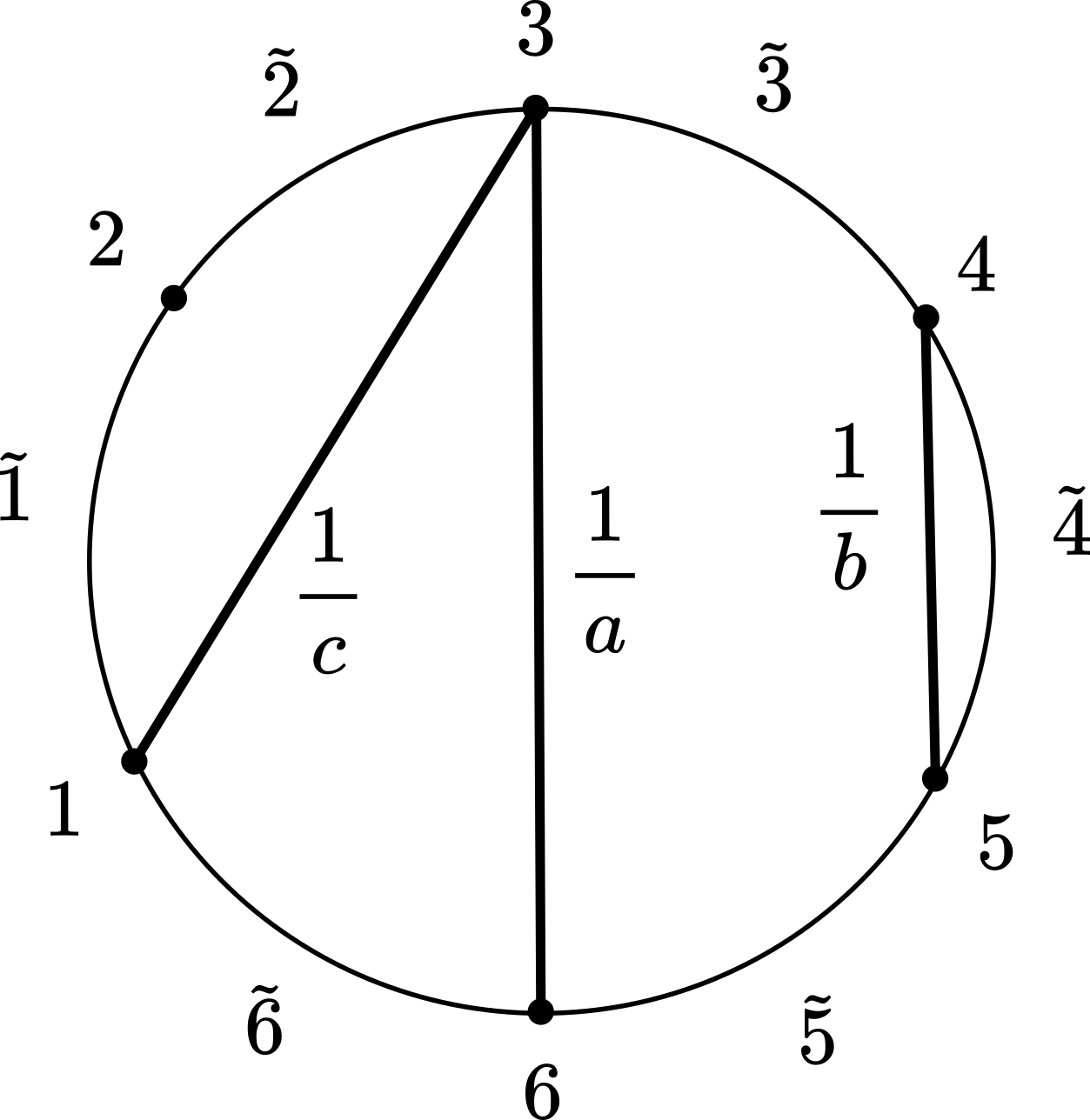}}
		\caption{The dual of the cactus network in Figure \ref{fig:cactus}.}\label{fig:dual}
	\end{figure}
	\begin{example}
		Figure \ref{fig:dual} shows the dual cactus network of the cactus network in Figure \ref{fig:cactus}.
	\end{example}
	
	\begin{remark}
		$\sigma^*$ is not the same as $\wt \sigma$. For example, for the planar electrical network on $[2]$ with an edge between $1$ and $2$, we have $\sigma=\{1\},\{2\}$, $\wt \sigma=\{\wt 1,\wt 2\}$ and $\sigma^*=\{\wt 1\},\{\wt 2\}$.  
	\end{remark} 
	A grove $F$ in $G$ corresponds to a dual grove $F^*$ in $G^*$ which consists of duals of edges not in $F$. We have $w(F)=w(F^*) \prod_{e \text{ edge of }G} c(e)$ and $\sigma(F^*)=s(\wt{\sigma(F)})$. 
	Therefore in $\R\P^{\Cat_n-1}$, duality is the homeomorphism given by $\Lambda_{s(\wt \sigma)} \mapsto \Lambda_\sigma$. Therefore duality is a continuous symmetry of $\mathcal R_n$, rather than merely piecewise continuous as suggested by the definition.
	
	We now explain how Lam's map relates to duality. Let ${\Sigma}:V \ra V$ denote the linear operator $(e_1,e_{\wt 1},\dots,e_{n},e_{\wt n}) \mapsto ((-1)^n e_{\wt n},e_1,e_{\wt 1},\dots,e_{\wt{n-1}},e_n)$. We define the \textit{cyclic shift operator} on $\Gr_{\geq 0}(n+1,2n)$ mapping $X$ to $X \cdot {\Sigma}$. 
	
	\begin{lemma}\la{lem:cyclic}
		Under Lam's map, duality becomes the cyclic shift operator.
	\end{lemma}
	
	\begin{proof}
		It follows from $\left(\Alt^{n+1}{\Sigma}\right) f_\sigma = f_{s(\wt \sigma)}$.
	\end{proof}
	
	Let $L^*$ denote the response matrix of $G^*$. We have the following relation between $L^*$ and $R$.
	\begin{proposition}[\cite{KW11}*{Proposition 2.9}]\la{prop:ltilde}
		$L^*_{ij}=\tfrac{1}{2} (R_{ij} + R_{i+1,j+1} - R_{i+1,j} - R_{i,j+1} )$.
	\end{proposition}

	\section{Proof of Theorem~\ref{thm:main}} \label{sec:proof_of_thm_main}
	
	Our goal is to show that $\IG^{\Omega}(n+1, 2n)$ is the intersection, inside $\R\P^{\binom{2n}{n+1}-1}$, of $\Gr(n+1,2n)$ with {$\mathcal{H}_n$.}  We will first show that there is a linear space $K$ in $\R^{\binom{2n}{n+1}}$ such that $\IG^{\Omega}(n+1,2n) = \Gr(n+1,2n) \cap \P(K)$.

	$\Omega$ is a skew symmetric pairing $V \times V \to \R$. We use $\Omega$ to induce a linear map $\kappa: \Alt^{n+1}(V) \to \Alt^{n-1}(V)$, defined as follows on simple tensors:
	\[ \kappa(v_1 \wedge v_2 \wedge \cdots \wedge v_{n+1}) = \sum_{1 \leq p < q \leq n+1} (-1)^{p+q-1} \Omega(v_p, v_q) \left(v_1 \wedge \cdots \wedge \widehat{v_p} \wedge \cdots \widehat{v_q} \wedge \cdots \wedge v_{n+1} \right),\] where the hats indicate omitted vectors. We set $K = \Ker \kappa$. 
	
	\begin{lemma} \label{lem:kappa_detects_isotropy}
		Let $v_1 \wedge \cdots \wedge v_{n+1}$ be a nonzero simple tensor in $\Alt^{n+1} V$. We have $\kappa(v_1 \wedge \cdots \wedge v_{n+1}) = 0$ if and only if $\Span(v_1, v_2, \ldots, v_{n+1})$ is isotropic with respect to $\Omega$.
	\end{lemma}
	
	\begin{proof}
		The condition that $v_1 \wedge \cdots \wedge v_{n+1} \neq 0$ is equivalent to imposing that the $v_i$ are linearly independent in $V$. 
		We deduce that the wedges $v_1 \wedge \cdots \wedge \widehat{v_p} \wedge \cdots \widehat{v_q} \wedge \cdots \wedge v_{n+1}$ are linearly independent in $\Alt^{n-1} V$. 
		Thus,  $\kappa(v_1 \wedge \cdots \wedge v_{n+1}) = 0$ if and only if $\Omega(v_p, v_q)=0$ for all $p$ and $q$; this is the same as saying that $\Span(v_1, v_2, \ldots, v_{n+1})$ is $\Omega$-isotropic.
	\end{proof}
	
	{
		\begin{corollary}\label{cor:Omega-int}
			$\IG^{\Omega}(n+1, 2n)=\Gr(n+1, 2n)\cap \P(K)$
		\end{corollary}
		\begin{proof}
			This follows immediately from the previous lemma and the fact that $\Gr(n+1, 2n)$ is the projectivization (in $\P \Alt^{n+1} V$) of the space of simple tensors.
		\end{proof}
	}
	
	

	Next we want to check that {$\mathcal{H}_n$,} the image of Lam's map, is contained in $K$.  We will do this by exploring the relationship of $f_\sigma$ and a set of vectors $\{v_k\}$ in $V$.  Let $(\sigma,\wt{\sigma})$ be a Kreweras pair with blocks $B_1$, $B_2$, \dots, $B_{n+1}$. For each block $B_k$, we define $v_k$ as follows: If $B_k \subseteq [n]$, then $v_k = \sum_{b \in B_k} (-1)^b e_b$; if $B_k \subseteq [\wt{n}]$ then $v_k = \sum_{\wt{b} \in B_k} (-1)^b e_{\wt{b}}$.  We will need two lemmas about these vectors.
	\begin{lemma}\label{lem:vecs-iso}
		For a Kreweras pair $(\sigma,\wt{\sigma})$, $\Span(v_1, v_2, \ldots, v_{n+1})$ is isotropic for $\Omega$.
	\end{lemma}

	\begin{proof}
		It is enough to check that $\Omega(v_i, v_j)=0$ for any $i$ and $j$. This is obvious unless there is some $b \in B_i$ and $c \in B_j$ with $\Omega(e_b, e_c) \neq 0$, so we may assume that such a pair $(b,c)$ exists. Without loss of generality, let $B_i \subseteq [n]$ and $B_j \subseteq [\wt{n}]$. There are two cases:
		\begin{enumerate}
			\item There are some $p < q$ such that $p \in B_i$, $\wt{p} \in B_j$, $\wt{q-1} \in B_j$ and $q \in B_i$.
			\item There are some $p \leq q$ such that $\wt{p-1} \in B_j$, $p \in B_i$, $q \in B_i$ and $\wt{q} \in B_j$.
		\end{enumerate}
		In both cases, there are no adjacent pairs of elements in $B_i$ and $B_j$ other than the listed ones. In the first case, if $p=q=1$, then
		\[ \Omega(v_i, v_j) = (-1)^{p+p} + (-1)^{n+1+n} = 0. \]
		Otherwise,
		\[ \Omega(v_i, v_j) = (-1)^{p+p} + (-1)^{q+(q-1)} = 0. \]
		In the second case, if $p=1$ then
		\[ \Omega(v_i, v_j) =(-1)^{q+q}+(-1)^{n+1+n}=0 . \]
		Otherwise,
		\[ \Omega(v_i, v_j) = (-1)^{p+(p-1)} + (-1)^{q+q} = 0 . \]
	\end{proof}
	
	\begin{lemma}\label{lem:f-sig-vecs}
		For a Kreweras pair $(\sigma,\wt{\sigma})$, $f_{\sigma}= \pm v_1 \wedge v_2 \wedge \cdots \wedge v_{n+1}$.
	\end{lemma}
	
	\begin{proof}
		It is clear that, when we expand out the wedge product, we will get a sum $\sum \pm e_{I, \tilde{I}}$, running over all pairs $(I, \tilde{I})$ concordant with $(\sigma, \tilde{\sigma})$. 
		The challenge is to check that all the $e_{I, \tilde{I}}$ come with the same sign.
		
		It is clearly enough to check that $e_{I, \tilde{I}}$ and $e_{J, \tilde{J}}$ come with the same sign where $I \sqcup \tilde{I}$ and $J \sqcup \tilde{J}$ differ by changing a single element.
		Suppose that $p$ and $q$ are elements of the block $B_k$, and that $J \sqcup \tilde{J}$ is obtained from $I \sqcup \tilde{I}$ by replacing $p$ with $q$.
		Without loss of generality, we may assume that $B_k \subseteq [n]$, that $p<q$ and that $p$ and $q$ are consecutive elements of $B_k$.
		
		There are two places where we get sign factors comparing the coefficients of $e_{I, \tilde{I}}$ and $e_{J, \tilde{J}}$. The first is that $e_p$ and $e_q$ come with coefficients $(-1)^p$ and $(-1)^q$, so we pick up a factor of $(-1)^{p-q}$ from there. 
		
		The second place is that we start by expanding the wedge product $v_1 \wedge v_2 \wedge \cdots \wedge v_{n+1}$ and then must reorder each term using the order $1 < \wt{1} < 2 < \wt{2} < \cdots < n < \wt{n}$. So we need to work out how many of the other factors $e_r$ have $r$ between $p$ and $q$ in this order; in other words, how many factors have $r \in \{ \wt{p}, p+1, \wt{p+1}, \cdots, q-1, \wt{q-1} \}$.
		Each block other than $B_k$ is either contained in $\{ \wt{p}, p+1, \wt{p+1}, \cdots, q-1, \wt{q-1} \}$, or else is disjoint from this interval.
		So we pick up a factor of $(-1)^{\ell}$ where $\ell$ is the number of blocks in this interval. In short, we need to show that $\ell \equiv q-p \bmod 2$.
		
		Define a non-crossing partition $(\tau, \tilde{\tau})$ of $\{ p, \wt{p}, p+1, \wt{p+1}, \cdots, q-1, \wt{q-1} \}$ by using the $\ell$ blocks from $(\sigma,\wt{\sigma})$, plus one more singleton block $\{ p \}$. Then $(\tau, \tilde{\tau})$ is a Kreweras pair for this $2(q-p)$ element set. Therefore, $\ell+1 = q-p+1$ and we deduce that $\ell = q-p$. In particular, $\ell \equiv q-p \bmod 2$.
	\end{proof}
	
	{
		\begin{corollary}\label{cor:image}
			$\mathcal{H}_n\subseteq K$
		\end{corollary}
		\begin{proof}
			Taken together, Lemmas~\ref{lem:kappa_detects_isotropy},~\ref{lem:vecs-iso}, and~\ref{lem:f-sig-vecs} imply that $\kappa(f_\sigma)=0$.  Since $\mathcal{H}_n$ is the set of linear combinations of the $f_\sigma$'s, we have $\mathcal{H}_n\subseteq K$.
		\end{proof}
	}
	
	
	We are now finally ready to prove our main theorem.
	\begin{proof}[Proof of Theorem~\ref{thm:main}]
		{From Corollary~\ref{cor:image}, we know} 
		that ${\mathcal{H}_n\subseteq}K$.  
		{By~\cite{Stanley}*{Example 4}, }the dimension of $K$ is $\Cat_n$.  This means ${\mathcal{H}_n=}K$.  {From Corollary~\ref{cor:Omega-int}, we have that $\IG^{\Omega}(n+1, 2n)=\Gr(n+1, 2n)\cap \P(K)=\Gr(n+1, 2n)\cap \mathcal{H}_n$.  Thus, we conclude, using Theorem~\ref{theorem:Lammain}, that $\IG^{\Omega}_{\geq 0}(n+1, 2n)=\mathcal{X}_n$.}
	\end{proof}

	\begin{remark}The row span of $XD$ is the space of pairs of harmonic and conjugate harmonic functions on the cactus network $G$. This space is identified with the space of discrete holomorphic functions (that is functions in the kernel of a Kasteleyn matrix) on a dimer model associated to $G$ by the generalized Temperley's bijection of Kenyon, Propp and Wilson \cite{KPW}. However the kernel of this Kasteleyn matrix constructed using the Kasteleyn sign in \cite{Ken02}*{Section 3.1} is not totally nonnegative because it does not satisfy the Kasteleyn sign condition at boundary faces. We can resolve this by modifying the Kasteleyn sign using a gauge transformation at boundary vertices, which {corresponds to} multiplying by the diagonal matrix $D$.
	\end{remark}

	\section{$U_{\text{not shorted}}$ and $U_{\text{connected}}$}
	
	Schubert charts for Lagrangian grassmannians are well understood, but the details for a degenerate form such as $\Omega$ or $\Omega^D$ are not that standard; we work them out here.
	We recall the linear space $K$ from Section~\ref{sec:proof_of_thm_main}; the image of Thomas Lam's map, spanned by the vectors $f_{\sigma}$. 
	We also reuse the notations $e_{I, \wt{I}}$ and $V$ from that section.
	
	We begin by proving a useful lemma:
	
	\begin{lemma} \label{lem:extreme_Deltas}
		For $v=\sum_{\sigma} \Lambda_\sigma f_\sigma$ in $K$, we have $\Delta_{[n], \{ \wt{1} \}}(v) = \Delta_{[n], \{ \wt{2} \}}(v) = \cdots = \Delta_{[n], \{ \wt{n} \}}(v)=\Lambda_{\{1\},\{2\},\dots,\{n\}}$ and $\Delta_{\{ 1 \}, [\wt{n}]}(v) = \Delta_{\{ 2 \}, [\wt{n}]}(v) = \cdots = \Delta_{\{ n \}, [\wt{n}]}(v)=\Lambda_{\{1,2,\dots,n\}}$.
	\end{lemma}
	
	\begin{proof}
		In the sum~\eqref{eq:f_sigma_defn} defining $f_{\sigma}$, all the terms $e_{I, \wt{I}}$ have $\#(I)$ equal to the number of blocks of $\sigma$ (since $I$ and $\sigma$ are concordant). Thus, terms of the form $e_{[n],\ \{\wt{k}\}}$ can only occur if $\sigma$ is the partition $\{ 1\},\{2\},\ldots, \{n \}$. Since all the $e_{[n],\ \{\wt{k}\}}$ occur with coefficient $1$ in $f_{\{ 1\},\{2\},\ldots, \{n \}}$, they all occur with the same coefficient $\Lambda_{\{ 1\},\{2\},\ldots, \{n \}}$ in $v$. 
		A similar argument applies to $e_{\{n\}, [ \wt{n} ]}$. 
	\end{proof} 
	
	Recall that $U_{\text{not shorted}}$ is the open set in $K$ where the $\Delta_{[n],\{\wt{k}\}}$ are nonzero and $U_{\text{connected}}$ is the open set in $K$ where the $\Delta_{\{ 1 \}, [\wt{n}]}$ are nonzero.  Recall also the operation $s$ and the cyclic shift ${\Sigma}$ from Section \ref{sec:dual}. Notice that $s([\wt n])=[ n]$ and $s(\{ k\})=\{\wt{k-1}\}$, so ${\Sigma}$ maps $U_{\text{connected}}$ to $U_{\text{not shorted}}$.

	We next prove that the names $U_{\text{not shorted}}$ and $U_{\text{connected}}$ are appropriate.
	
	\begin{lemma}\label{lem:conshort}
		The totally nonnegative points of $U_{\mathrm{not\  shorted}}$ correspond to planar electrical networks and those of $U_{{\rm connected}}$ correspond to cactus networks where the underlying graph is connected.
	\end{lemma}
	\begin{proof}
		The totally nonnegative points of $U_{\mathrm{not\  shorted}}$ correspond to cactus networks with $\Lambda_{\{1\},\{2\},\dots,\{n\}} >0$. Therefore the cactus network has a grove with each vertex in $[n]$ in a different component, which implies that no two vertices are shorted. The totally nonnegative points of $U_{\mathrm{connected}}$ correspond to cactus networks with $\Lambda_{\{1,2,\dots,n\}} >0$. This means that the cactus network has a spanning tree and therefore it is connected.
	\end{proof}
	
	We are now ready to prove Theorem~\ref{thm:SchubertChart}, which gives a matrix form for elements of $U_{\text{not shorted}}$ and $U_{\text{connected}}$.
	
	\begin{proof}[Proof of Theorem~\ref{thm:SchubertChart}]
		We will first verify the form for matrices in $U_{\text{not shorted}}$ that is claimed in Theorem~\ref{thm:SchubertChart}; the analogous claim for matrics in $U_{\text{connected}}$ will then follow from cyclic symmetry. 
		
		Let $X$ be an $\Omega$-isotropic subspace in $U_{\text{not shorted}}$ and let $M$ be an $(n+1) \times (2n)$ matrix so that the rows of $MD$ span $X$. 
		So the rows of $M$ are isotropic for $\Omega^D$. 
		
		Since the Pl\"ucker coordinates {$\Delta_{[n],\ \{\wt{k}\}}(M)$} are nonzero, the columns indexed by {$[n]$} are linearly independent and we can normalize them to be standard basis vectors as in Theorem~\ref{thm:SchubertChart}. This fixes $M$ up to left multiplication by {nonsingular} matrices of the form $\begin{sbm} \ast & 0 \\ \ast & \mathrm{Id}_n \end{sbm}$. 
		{Since $\Delta_{[n],\ \{\wt{1}\}}(M)=\Delta_{[n],\ \{\wt{2}\}}(M)=\dots=\Delta_{[n],\ \{\wt{n}\}}(M)$ by Lemma~\ref{lem:extreme_Deltas}, the top row of $M$ must then be must be a scalar multiple of $(0,1,0,1,\cdots)$, and we can fix that scalar to be $1$.} We have now fixed $M$  up to left multiplication by matrices of the form $\begin{sbm} 1 & 0 \\ \ast & \mathrm{Id}_n \end{sbm}$ or, in other words, up to adding multiples of the top row to the other rows. So far, we know that $M$ is of the form 
		\be \label{mat:s}
		\begin{bmatrix}
			0&1&0&1&\cdots&0&1 \\
			1&S_{11}&0&S_{12}&\cdots&0&S_{1n} \\
			0&S_{21}&1&S_{22}&\cdots&0&S_{2n} \\
			\vdots&\vdots&\vdots&\vdots&\cdots&\vdots&\vdots\\
			0&S_{n1}&0&S_{n2}&\cdots&1&S_{nn} \\
		\end{bmatrix}.  
		\ee
		and that the matrix $S$ is unique up to adding a multiple of $(1,1,\ldots,1)$ to each row.  It remains to show that $S$ has the required symmetry property.
		
		We now impose the condition that the rows of $M$ are $\Omega^D$ isotropic. The first row is in the kernel of $\Omega^D$, so this is automatic. Look at the $(p+1)$-st and the $(q+1)$-st row, for $p<q$. We compute that the pairing of these rows under $\Omega^D$ is $S_{pq} - S_{p(q-1)} + S_{q(p-1)} - S_{qp}$. 
		We can reorder this as  $S_{p(q-1)}-S_{pq} =S_{q(p-1)}-S_{qp}$. In other words, the rows of $M$ are $\Omega^D$ isotropic if and only if the matrix $S_{i(j-1)}-S_{ij}$ is symmetric, as required.

		Suppose $X$ is an $\Omega$-isotropic subspace in $U_{\text{connected}}$. $X\cdot \Sigma$ is then an $\Omega$-isotropic subspace in $U_{\text{not shorted}}$, so it can be represented by a matrix $MD$, where $M$ has the form $(\ref{mat:s})$. $D\Sigma^{-1} D^{-1}$ is the linear transformation $(e_1,e_{\wt 1},e_2,e_{\wt 2},\dots,e_{n},e_{\wt n}) \mapsto (e_{\wt 1},-e_2,e_{\wt 2},-e_3,\dots,e_{\wt{n}},-e_{1})$, so we get
		$$MD\Sigma^{-1}=
		\begin{bmatrix} 
			-1&0&-1&0&\cdots&-1&0 \\
			T_{11}&1&T_{12}&0&\cdots&T_{1n}&0 \\
			T_{21}&0&T_{22}&1&\cdots&T_{2n}&0 \\
			\vdots&\vdots&\vdots&\vdots&\cdots&\vdots&\vdots\\
			T_{n1} & 0&T_{n2}&0 & \cdots & T_{nn}&1 \\
		\end{bmatrix} D,$$ 
		where $T_{ij}=-S_{i(j-1).}$

	\end{proof}
	
	We can now prove Theorems~\ref{thm:Lmatrix_chart} and~\ref{thm:Rmatrix_chart}, which tell us how the matrices from Theorem~\ref{thm:SchubertChart} relate to the response and effective resistance matrices.

	\begin{proof}[Proof of Theorem~\ref{thm:Lmatrix_chart}]
		Let $G$ be a planar electrical network with response matrix $L$. By Lemma \ref{lem:conshort}, Lam's map associates to $G$ a totally nonnegative point in $U_{\text{not shorted}}$. By Theorem~\ref{thm:SchubertChart}, this point is the row span of a matrix $M D$, where $M$ is of the form (\ref{mat:s}). We have {for $i \neq j$}, $\Delta_{[n],\{\tilde i\}}(MD)= (-1)^{{n \choose 2}+1}$ and $\Delta_{ [n] \setminus \{j\},\{ \widetilde{i-1},\tilde i\}}(MD)=(-1)^{{n \choose 2}+1}(S_{j(i-1)}-S_{j i})$ (with indices periodic modulo $n$), therefore using (\ref{DeltaRatios}), we get $S_{j(i-1)}-S_{ji}=L_{ij}$.  {
			Since all non-diagonal entries of the two matrices are the same and they both have rows that sum to 0, the matrices are the same.}
	\end{proof}

	\begin{proof}[Proof of Theorem~\ref{thm:Rmatrix_chart}]
		Suppose $G$ is a connected cactus network with effective resistance matrix $R$. Using Lemma \ref{lem:conshort}, we have that Lam's map associates to $G$ a totally nonnegative point $X$ in $U_{\text{connected}}$. Therefore Lam's map associates to the dual cactus network $G^*$ a totally nonnegative point in $U_{\text{not shorted}}$, which by Lemma \ref{lem:cyclic} is $X \cdot {\Sigma}$. By Theorem \ref{thm:Lmatrix_chart} and Proposition \ref{prop:ltilde}, $X\cdot {\Sigma}$ is represented by a matrix $MD$, where $M$ has the form 
		(\ref{mat:s}) with $T_{j(i+1)}-T_{ji}=S_{j(i-1)}-S_{ji}=L^*_{ij}$.
	\end{proof}

	\section{Comparison other work on Lagrangian Grassmannians.}\label{sec:Karpman}
	
	We aware of two earlier studies of total positivity for Lagrangian Grassmannians. 
	In this section, we explain why we think this work is not very close to ours, but suggest that there might be room for a common generalization.
	Since this section is an overview, we will omit many proofs.
	
	\subsection{Work of Karpman}	
	We first discuss the work of Karpman~\cite{Karpman1, Karpman2}. 
	Karpman studies totally nonnegative spaces in $\R^{2m}$ which are isotropic with respect to the 
	non-degenerate skew form 
	\[ \langle e_i, e_j \rangle = \begin{cases} (-1)^j & i+j = 2m+1 \\ 0 & \text{otherwise.} \\ \end{cases}  \]
	Let's denote this space $\LG^{\Karp}(m,2m)$.
	
	Recall that our space of electrical networks is contained in $\IG^{\Omega}(n+1,2n) \cong \LG(n-1, 2n-2)$. 
	So one might imagine relating $\IG^{\Omega}(n+1, 2n)$ either to $\LG^{\Karp}(n-1, 2n-2)$ or, perhaps, to $\LG^{\Karp}(n,2n)$.
	In either case, we do not see such a relationship. 
	
	The space of electrical networks has an $n$-fold rotational symmetry. 
	In contrast, $\LG^{\Karp}(m,2m)$ does not have a rotational symmetry. 
	For example, $\LG^{\Karp}(2,4)$ has three codimension one strata, two of which are triangles and one of which is a quadilateral, so they cannot be permuted by a $3$-fold symmetry.
	By contrast, the three codimension one strata of $\IG^{\Omega}(4,6)$ are all quadrilaterals, and are permuted cyclically by the rotational symmetry of electrical networks.
	
	Moreover, the enumeration of cells does not match. Both $\LG^{\Karp}(n-1,2n-2)$ and $\IG^{\Omega}(n+1, 2n)$ have  dimension $\binom{n}{2}$ and $n$ codimension-$1$ cells.
	However  (for $n \geq 3$), $\LG^{\Karp}(n-1, 2n-2)$ has $\binom{n+1}{2}-1$ codimension $2$ cells and $\IG^{\Omega}(n+1,2n)$ has $\binom{n+1}{2}$ codimension $2$ cells.

	\subsection{Gaussoids}
	Boege, D'Ali, Kahle and Sturmfels introduced the study of \emph{gaussoids}~\cite{gaussoids}, motivated by problems in algebraic statistics. 
	This work also involves considering subspaces of the Lagrangian Grassmannian whose Pl\"ucker coordinates obey a sign condition.
	However, we will argue in this section that the sign condition they consider is very different from either Karpman's or ours, and is likely not related to the theory of total positivity at all.
	
	It is common in statistics to study $m$ quantities and summarize the correlations between them in a symmetric matrix known as the \emph{covariance matrix}.
	Boege, D'Ali, Kahle and Sturmfels embed $m \times m$ symmetric matrices in $G(m,2m)$ by sending the symmetric matrix $\Sigma$ to the row span of the $m \times (2m)$ matrix $[ \mathrm{Id} \ \Sigma]$.
	This $m$-plane is Lagrangian for the skew form $\left[ \begin{smallmatrix} 0 & \mathrm{Id}_m \\ - \mathrm{Id}_m & 0 \\ \end{smallmatrix} \right]$.
	We'd rather choose an embedding which is Lagrangian for Karpman's form. In order to do this, we define an $m \times 2m$ matrix $A(\Sigma)$ by
	\[
	A(\Sigma)_{ij} = \begin{cases} (-1)^{i-1} & i+j=m+1 \\ 0 & j \leq m,\ i+j \neq m+1 \\ \sigma_{i(j-m)} & j \geq m \end{cases}.
	\]
	We depict the example $m=4$ below.
	\[
	A(\Sigma) = \begin{bmatrix} 
		0&0&0&1&\sigma_{11}&\sigma_{12}&\sigma_{13}&\sigma_{14} \\
		0&0&-1&0&\sigma_{21}&\sigma_{22}&\sigma_{23}&\sigma_{24} \\
		0&1&0&0&\sigma_{31}&\sigma_{32}&\sigma_{33}&\sigma_{34} \\
		-1&0&0&0&\sigma_{41}&\sigma_{42}&\sigma_{43}&\sigma_{44} \\
	\end{bmatrix} 
	\]

	Boege \emph{et al} impose that $\Sigma$ is positive definite, meaning that all of its principal minors are nonnegative. 
	We now describe the corresponding condition in terms of the Pl\"ucker coordinates $\Delta_I(A(\Sigma))$.
	For $I \subset [2m]$, set $\overline{I} = \{ 2m+1-i : i \in [2m] \setminus I \}$.
	The matrix $\Sigma$ is positive definite if and only if $\Delta_I(A(\Sigma))>0$ for all $I$ such that $I = \overline{I}$. 
	(These are the face labels which occur on the ``spine" in the sense of~\cite{Karpman2}.)
	This condition forms a dense open subset of $\LG^{\Karp}_{\geq 0}$: a cell of Karpman's space indexed by bounded affine permutation $f$ obeys this condition if and only if $f(\{1,2,\ldots, m \}) = \{ m+1, m+2, \ldots, 2m \}$. 
	So, a dense open set in $\LG^{\Karp}(m,2m)_{\geq 0}$ gives rise to gaussoids, with certain additional positivity conditions coming from the other Pl\"ucker coordinates.
	
	There is a notion of a ``positive gaussoid", motivated by ideas from algebraic statistics. 
	However, we will see that this notion of positivity is not the one obtained by asking that the Pl\"ucker coordinates of $A(\Sigma)$ be nonnegative, nor can this be fixed by any simple rearrangement of columns or signs. 
	
	Let $K \subset [m]$ and let $i$ and $j$ be distinct elements of $[m] \setminus K$. 
	The \emph{almost principal minor} $a_{ij|K}$ is the minor of $\Sigma$ whose rows are indexed by $\{ i \} \cup K$ and whose columns are indexed by $\{ j \} \cup K$, where the 
	the elements of $K$ are listed after $i$ and after $j$, and $K$ is ordered the same way in both cases.
	For example, $a_{13|2} = \sigma_{13} \sigma_{22} - \sigma_{12} \sigma_{23}$. 
	A gaussoid is \emph{positive} if all its \emph{almost principal minors} are nonnegative.
	We note that this is a different sign condition than the one we get by imposing that $\Delta_{1267}(A(\Sigma)) > 0$ (here we take $m=4$, to match our example above); we have $\Delta_{1267} = \sigma_{12} \sigma_{33} - \sigma_{13} \sigma_{22}$. 
	
	This is not a minor technicality; we will now show that there is no way to embed symmetric matrices into the Grassmannian such that the positivity conditions from the positive Grassmannian correspond to the positivity conditions from gaussoids.
	More precisely, let $b_1$, $b_2$, \dots, $b_{2m}$ be some permutation of $[2m]$ and let $\delta_i$ and $\epsilon_{ij}$ be elements of $\{ \pm 1 \}$, indexed by $1 \leq i, j \leq m$.
	Given a symmetric matrix $\Sigma$, form a linear space
	\[ B(\Sigma)_{ij} = 
	\begin{cases} \delta_i & j = b_i \\ 0 & j \in \{ b_1, b_2, \ldots,  b_m \} \setminus \{ b_i \} \\ \epsilon_{ik} \sigma_{ik} & j = b_{k+m} \\ \end{cases} \]
	In other words, the columns $\{ b_1, b_2, \ldots, b_m \}$ contain a signed permutation matrix, with the order of $(b_1, \ldots, b_m)$ encoding the permutation and the $\delta_i$ encoding the signs. 
	The columns  $\{ b_{m+1}, b_{m+2}, \ldots, b_{2m} \}$ contain $\pm$ the entries of $\Sigma$, with the columns reordered according to the order of $(b_{m+1}, b_{m+2}, \ldots, b_{2m})$ and with signs given by the $\epsilon_{ij}$. 
(We don't need to consider reordering the rows of $\Sigma$, as we could always put them back in order by left multiplying by a permutation matrix.)
	
	\begin{theorem}
		For $m \geq 3$, there is no choice of $b_1$, $b_2$, \dots, $b_{2m}$,  $\delta_i$ and $\epsilon_{ij}$  such that the principal and almost principal minors occur as a subset of the Pl\"ucker coordinates of $\Sigma$. 
	\end{theorem}
	
	\begin{proof}
		Start with some choice of $b_1$, $b_2$, \dots, $b_{2m}$,  $\delta_i$ and $\epsilon_{ij}$; we show that we cannot obtain all of the principal and almost principal minors with their correct signs.
		To begin with, consider the $2 \times 2$ principal minor $\sigma_{aa} \sigma_{bb} - \sigma_{ab} \sigma_{ba}$.
		The condition that $\sigma_{aa} \sigma_{bb}$ appears with the opposite sign from $\sigma_{ab} \sigma_{ba}$ implies that $\epsilon_{aa} \epsilon_{bb} = \epsilon_{ab} \epsilon_{ba}$.
		Similarly, from the almost principal minor,  $\sigma_{ab} \sigma_{cc} - \sigma_{ac} \sigma_{cb}$, we obtain that  $\epsilon_{ab} \epsilon_{cc} = \epsilon_{ac} \epsilon_{cb}$. 
		
		These relations force there to be some signs $\alpha_1$, \dots, $\alpha_m$, $\beta_1$, \dots, $\beta_m$ in $\{ \pm 1 \}$ such that $\epsilon_{ab} = \alpha_a \beta_b$. 
		Then we can multiply row $a$ by $\alpha_a$, which will preserve all Pl\"ucker coordinates up to a global sign, so we can assume that $\alpha_1=\dots=\alpha_m=1$.  
		In short, we have reduced to the situation that the columns of $B(\Sigma)$ are, in some order, $\pm$ the columns of $\Sigma$ and $\pm$ the standard basis vectors, where the signs of the columns are given by $\beta_j$.
		
		We now look at the $1 \times 1$ minors of $\Sigma$. Every $1 \times 1$ minor is either principal or almost principal, so our condition is that each $\sigma_{ij}$ occurs with positive sign as a minor of $B(\Sigma)$. 
		Fix a row $i$ and consider the condition that $\sigma_{i j_1}$ and $\sigma_{i j_2}$ occur with the same sign.
		By switching the names $j_1$ and $j_2$, we may assume that $b_{j_1} < b_{j_2}$.
		We have 
		\[ \Delta_{b_1 b_2 \cdots b_{i-1} b_{i+1} \cdots b_m b_{j_a}} = \pm \sigma_{i j_a}  \]
		where the sign comes from (1) the partial permutation matrix in columns $\{ b_1, b_2, \cdots b_{i-1}, b_{i+1}, \cdots, b_m \}$, (2) the $\delta$'s and (3) the factor $\beta_{j_a}$. If we take the ratio of these formulas for $j_1$ and $j_2$, almost all the factors cancel and we deduce that
		\[ \beta_{j_1}/\beta_{j_2} = (-1)^{\# \{ i' \in [m] : i' \neq i,\ b_{m+j_1} < b_{i'} < b_{m+j_2} \}} . \]
		The left hand side is independent of $i$, so the right hand side must be as well, and we deduce that, for any $j_1$ and $j_2$, either all of $\{ b_1, \ldots, b_m \}$ lie between $b_{m+j_1}$ and $b_{m+j_2}$, or else none of them do. So the column indices $\{ b_1, \ldots, b_m \}$ are a cyclically consecutive subset of $[2m]$.
		Using the dihedral symmetry of the totally nonnegative Grassmannian, we may now assume that $\{ b_1, \ldots, b_m \} = [m]$.
		
		We now look at the $2 \times 2$ principal and almost principal minors. 
		Choose three indices $j$, $k$, $\ell$, with $b_j < b_k < b_{\ell}$.
		Consider the minors $\sigma_{jj} \sigma_{kk} - \sigma_{jk} \sigma_{kj}$,  $\sigma_{jj} \sigma_{k\ell} - \sigma_{j\ell} \sigma_{kj}$ and  $\sigma_{j\ell} \sigma_{kk} - \sigma_{jk} \sigma_{k\ell}$.  Up to sign, these are the Pl\"ucker coordinates  $\Delta_{\{b_1 b_2 \cdots \cdots b_m, b_{j+m}, b_{k+m}\}\setminus\{b_j,b_k\}}$, $\Delta_{\{b_1 b_2 \cdots \cdots b_m, b_{k+m}, b_{\ell+m}\}\setminus\{b_j,b_k\}}$, and $\Delta_{\{b_1 b_2 \cdots \cdots b_m, b_{j+m}, b_{\ell+m}\}\setminus\{b_j,b_k\}}$.  

If the principal and almost principal minors appear as Pl\"ucker coordinates, then there is a choice of $\beta_j,\beta_k,\beta_\ell$ such that the maximal minors of the above matrix are exactly (not just up to sign) the minors of $\Sigma$ listed above.  Computing these minors and cancelling out common sign factors, we obtain that $-\beta_j=\beta_k=\beta_\ell$.  However, we can do the same analysis with the principal and almost principal minors $\sigma_{kk} \sigma_{\ell j} - \sigma_{kj} \sigma_{\ell k}$, $\sigma_{kj} \sigma_{\ell \ell} - \sigma_{k\ell} \sigma_{\ell j}$ and $\sigma_{kk} \sigma_{\ell \ell} - \sigma_{k\ell} \sigma_{\ell k}$. From this, we deduce that $\beta_j = \beta_k = -\beta_{\ell}$, and we obtain a contradiction.
	\end{proof}

	\subsection{Possibility of a synthesis?}
	
	We have seen that Karpman's study of total positivity for Lagrangian Grassmannians is not the same as that for electrical networks.
	It might be interesting, though, to ask whether there is a common framework which accommodates both subjects. 
	Let $K$ be any skew symmetric pairing $\R^m \times \R^m \to \R$. 
	The condition that a $k$-plane $V$ be isotropic for $K$ is a linear condition on the corresponding point of the Grassmannian $G(k,n)$, written in Pl\"ucker coordinates; let's write $H$ for the corresponding plane in $\R \P^{\binom{n}{k}-1}$.
	Here are some natural questions to ask, which we have seen have good answers for both the forms $\Omega$, from electrical networks, and for Karpman's form $\langle \ , \ \rangle$.
	
	\begin{enumerate}
		\item Under what circumstances does $H$ cross the cells of the totally nonnegative Grassmannian, $G(k,n)_{\geq 0}$, transversely?
		\item Under what circumstances is the intersection of $H$ with each cell either empty or else an open ball?
		\item Can one give a simple description, in terms of the combinatorics of bounded permutations, for when that intersection is nonempty?
	\end{enumerate}
	
	One could also ask these questions for linear subspaces of Pl\"ucker space other than those coming from skew forms.
	For example, the condition that a $k$-plane in $\R^n$ contains a given vector is also a linear condition in Pl\"ucker coordinates.
	
	Finally, we have seen that positive Gaussoids are very different; they do not come from $G(k,n)_{\geq 0}$ at all.
	One could wonder whether the nice behavior of positive Gaussoids suggests that there are other sign patterns of Pl\"ucker coordinates, besides positivity, which might be interesting to study.
	Alternatively, the nice behavior of positive matrices might make one wonder whether there is any interest in statistical models where the minors of $\Sigma$ have signs corresponding to totally positive points of $\LG(m,2m)$.

	\bibliographystyle{ams}
	\bibliography{references}
	
\end{document}